\documentclass[12pt,leqno,draft]{article}
\usepackage{amsfonts}
\usepackage{amsmath, amsthm, amsfonts, amssymb, color}
\usepackage{mathrsfs}
\usepackage{color}
\newcommand{\Div}{\mathrm{div}}

\newtheorem{thm}{Theorem}[section]

\theoremstyle{definition}

\newcommand{\scr}[1]{\mathscr #1}
\definecolor{wco}{rgb}{0.5,0.2,0.3}

\numberwithin{equation}{section} \theoremstyle{remark}

\newcommand{\ua}{\uparrow}

\title{{\bf Existence and
 Ergodicity for Nonlinear Fokker-Planck Equations for Probability Density Functions  }\footnote{Supported in part by the National Key R\&D Program of China (2022YFA1006000), NNSFC(12531007,12271398)
 and Deutsche Forschungsgemeinschaft (DFG, German Research Foundation) (Project-ID 317210226, SFB 1283).} }
\author{
{\bf Xing Huang$^{a)}$, Michael R\"ockner$^{b),c),d)}$,    Feng-Yu Wang$^{a)}$  }\\
\footnotesize{  a)  Center for Applied Mathematics and KL-AAGDM, Tianjin
University, Tianjin 300072, China} \\
\footnotesize{  b) Faculty of Mathematics, Bielefeld University, 33615 Bielefeld, Germany}\\
\footnotesize{ c) Academy of Mathematics and System Sciences, Chinese Academy of Sciences, Beijing, China}\\
\footnotesize{  d) School of Data Science, The Chinese University of Hong Kong, Shenzhen, China}\\
\footnotesize{xinghuang@tju.edu.cn, roeckner@math.uni-bielefeld.de, wangfy@tju.edu.cn  }\\
}
\begin{document}
\allowdisplaybreaks
\def\R{\mathbb R}  \def\ff{\frac} \def\ss{\sqrt} \def\B{\mathbf
B} \def\W{\mathbb W}
\def\N{\mathbb N} \def\kk{\kappa} \def\m{{\bf m}}
\def\ee{\varepsilon}\def\ddd{D^*}
\def\dd{\delta} \def\DD{\Delta} \def\vv{\varepsilon} \def\rr_V{\rho}
\def\<{\langle} \def\>{\rangle} \def\GG{\Gamma} \def\gg{\gamma}
  \def\nn{\nabla} \def\pp{\partial} \def\E{\mathbb E}
\def\d{\text{\rm{d}}} \def\bb{\beta} \def\aa{\alpha} \def\D{\scr D}
  \def\si{\sigma} \def\ess{\text{\rm{ess}}}
\def\beg{\begin} \def\beq{\begin{equation}}  \def\F{\scr F}
\def\Ric{\text{\rm{Ric}}} \def\Hess{\text{\rm{Hess}}}
\def\e{\text{\rm{e}}} \def\ua{\underline a} \def\OO{\Omega}  \def\oo{\omega}
 \def\tt{\tilde} \def\Ric{\text{\rm{Ric}}}
\def\cut{\text{\rm{cut}}} \def\P{\mathbb P} \def\ifn{I_n(f^{\bigotimes n})}
\def\C{\scr C}      \def\aaa{\mathbf{r}}     \def\r{r}
\def\gap{\text{\rm{gap}}} \def\prr{\pi_{{\bf m},\varrho}}  \def\r{\mathbf r}
\def\Z{\mathbb Z} \def\vrr{\varrho} \def\ll{\lambda}
\def\L{\scr L}\def\Tt{\tt} \def\TT{\tt}\def\II{\mathbb I}
\def\i{{\rm in}}\def\Sect{{\rm Sect}}  \def\H{\mathbb H}
\def\M{\scr M}\def\Q{\mathbb Q} \def\texto{\text{o}} \def\LL{\Lambda}
\def\Rank{{\rm Rank}} \def\B{\scr B} \def\i{{\rm i}} \def\HR{\hat{\R}^d}
\def\to{\rightarrow}\def\l{\ell}\def\iint{\int}
\def\EE{\scr E}\def\Cut{{\rm Cut}}
\def\A{\scr A} \def\Lip{{\rm Lip}}
\def\BB{\scr B}\def\Ent{{\rm Ent}}\def\L{\scr L}
\def\R{\mathbb R}  \def\ff{\frac} \def\ss{\sqrt} \def\B{\mathbf
B}
\def\N{\mathbb N} \def\kk{\kappa} \def\m{{\bf m}}
\def\dd{\delta} \def\DD{\Delta} \def\vv{\varepsilon} \def\rr{\rho}
\def\<{\langle} \def\>{\rangle} \def\GG{\Gamma} \def\gg{\gamma}
  \def\nn{\nabla} \def\pp{\partial} \def\E{\mathbb E}
\def\d{\text{\rm{d}}} \def\bb{\beta} \def\aa{\alpha} \def\D{\scr D}
  \def\si{\sigma} \def\ess{\text{\rm{ess}}}
\def\beg{\begin} \def\beq{\begin{equation}}  \def\F{\scr F}
\def\Ric{\text{\rm{Ric}}} \def\Hess{\text{\rm{Hess}}}
\def\e{\text{\rm{e}}} \def\ua{\underline a} \def\OO{\Omega}  \def\oo{\omega}
 \def\tt{\tilde} \def\Ric{\text{\rm{Ric}}}
\def\cut{\text{\rm{cut}}} \def\P{\mathbb P} \def\ifn{I_n(f^{\bigotimes n})}
\def\C{\scr C}      \def\aaa{\mathbf{r}}     \def\r{r}
\def\gap{\text{\rm{gap}}} \def\prr{\pi_{{\bf m},\varrho}}  \def\r{\mathbf r}
\def\Z{\mathbb Z} \def\vrr{\varrho} \def\ll{\lambda}
\def\L{\scr L}\def\Tt{\tt} \def\TT{\tt}\def\II{\mathbb I}
\def\i{{\rm in}}\def\Sect{{\rm Sect}}  \def\H{\mathbb H}
\def\M{\scr M}\def\Q{\mathbb Q} \def\texto{\text{o}} \def\LL{\Lambda}
\def\Rank{{\rm Rank}} \def\B{\scr B} \def\i{{\rm i}} \def\HR{\hat{\R}^d}
\def\to{\rightarrow}\def\l{\ell}\def\BB{\mathbb B}
\def\8{\infty}\def\I{1}\def\U{\scr U} \def\n{{\mathbf n}}\def\v{V}\def\LL{{\bf L}}\def\rr{\rho}
\maketitle

\begin{abstract}  For a class of nonlinear Fokker-Planck equations for probability densities,
we construct  distributional solutions by  using an approximation scheme with nonlinear Neumann problems in balls. Moreover, we introduce nonlinear functional inequalities to estimate   the convergence   rates  of the distributional solutions    to    the stationary solution, with respect to   the $p$-variance for $p\in [1,2]$ including the relative entropy $(p=1)$ and the variance $(p=2)$.  These nonlinear functional inequalities are then established in both non-degenerate and degenerate settings, so that the main results are applied
   to  a number of typical models with  various   convergence rates.

 \end{abstract} \noindent

\section{Introduction}

 Let $V\in C^2(\R^d)$ with   $Z(V):=\int_{\R^d}\e^{-V(x)}\d x<\infty,$ such that
 $$\mu(\d x):=\ff 1 {Z(V)} \e^{-V(x)}\d x=:\rr_V(x)\d x$$
 is a probability measure on $\R^d$.   Let $\d x$ denote the Lebesgue measure on $\R^d$ and set
 $$\D_1^+:= \bigg\{\rr\in L^1(\R^d):\ \rr\ge 0,\ \int_{\R^d} \rr(x)\d x=1\bigg\},$$
 $$\D_{1,\mu}^+:= \bigg\{ u\in L^1(\mu):\ u\ge 0,\ \int_{\R^d} u(x)\mu(\d x)=1\bigg\}.$$
 We consider the following nonlinear PDE for   $\rr_t\in \D_1^+$ and $u_t:=\rr_t/\rr_V\in \D_{1,\mu}^+$:
 \beq\label{E1} \pp_t \rr_t  ={\rm div} \Big\{\big(\phi'(\rr_t)+ \psi'(u_t) \big) \rr_V \nn u_t \Big\},  \ \ \ t\ge 0,\end{equation}
 where  $ \phi,\psi\in   C^2([0,\infty); [0,\infty)) $ satisfy
$$   \phi(0)=\psi(0)=0,\  \phi', \psi'\ge 0,\ \ \text{and \ either \ } \phi'|_{(0,\infty)}>0\ \text{or}\ \psi'|_{(0,\infty)}>0.  $$

PDE \eqref{E1} is indeed a nonlinear Fokker-Planck-Kolmogorov equation. To see this, let $\ell_\eta$ denote the distribution density function of an $\R^d$-valued random variable $\eta$ with respect to $\d x$, and consider the
 following density dependent (Nemytskii-type) SDE on $\R^d$:
 \beq\label{E1'} \beg{split} \d X_t=&\, \bigg(\ff{2[\phi(\ell_{X_t}) +\rr_V \psi(\ell_{X_t}/\rr_V)]}  {\ell_{X_t}}\bigg)^{\ff 1 2}(X_t)\,\d W_t\\
 &\,-\bigg[\bigg(\phi'(\ell_{X_t}) {+}\ff{\rr_V \psi\big(\ell_{X_t}/\rr_V\big)}{\ell_{X_t}}\bigg) \nn V \bigg](X_t) \d t,\ \ t\ge 0,\end{split}\end{equation}
 where $W_t$ is the $d$-dimensional Brownian motion on a complete filtrated probability space $(\OO,\F,\{\F_t\}_{t\ge 0},\P)$.
 Then for any probabilistically weak solution to \eqref{E1'} with $\rr_t:=\ell_{X_t}$ satisfying \eqref{conps} below, $\rr_t$ solves \eqref{E1}. This is an easy consequence of It\^o's formula. And vice versa, any distributional solution in the sense of Definition \ref{D1.2} below leads to a (probabilistically) weak solution to \eqref{E1'}. We refer to \cite{BR2020} and \cite{BR2024LNM} for details.

 When $\phi=0$  this equation covers the generalized porous media equation
  \beq\label{E1-1} \pp_t u_t= L_V \psi(u_t),\ \ \ L_V:= \DD-\nn V\cdot\nn\end{equation}
   for probability density functions $u_t$ with respect to $\mu$,
   and when $\psi=0$ it reduces to
  \beq\label{E1-2} \pp_t \rr_t=  \DD  \phi(\rr_t) + {\rm div} \big(\rr_t\phi'(\rr_t)\nn V\big),\ \ t\ge 0. \end{equation}

Heuristically, for any $f\in C_0^\infty(\R^d)$,   the  integration by parts formula gives
  \beg{align*} & \int_{\R^d}   \Big(f {\rm div} \Big\{\big(\phi'(\rr_t)+ \psi'(u_t) \big) \rr_V \nn u_t \Big\}\Big)(x)\d x\\
    &= -\int_{\R^d} \Big\<\nn f, \big(\phi'(\rr_t)+ \psi'(u_t) \big) \rr_V \nn u_t \Big\>(x)\d x\\
 &= \int_{\R^d} \Big(\big(\phi(\rr_t) +\rr_V\psi(u_t)   \big)  \DD f -   \big(\rr_t\phi'(\rr_t) +\psi(u_t)  \rr_V\big) \<\nn V,\nn f\>   \Big)(x) \d x.\end{align*}
 So,    we define the distributional solution to \eqref{E1}   as follows.

\beg{defn}[Distributional solution]\label{D1.2}
 A distributional solution to \eqref{E1} is a measurable map
 $$\rr: [0,\infty)\times \R^d\to [0,\infty)$$
 such that $ \rr_t \in \D_1^+$   for $t\ge 0$,
\begin{align}\label{conps}|\phi(\rr_t)+\rr_V \psi(u_t)|+|\rr_t\phi'(\rr_t) +\psi(u_t)  \rr_V| \in L_{loc}^1(\R^d),
  \end{align}
  and  the following integral equation holds for any $f\in C_0^\infty(\R^d)$:
 \beq\label{DB} \beg{split} & \int_{\R^d} \big(f\rr_t-f\rr_0\big)(x)\d x \\
 &=   \int_0^t \d s\int_{\R^d} \Big(\big(\phi(\rr_s) +\rr_V\psi(u_s)   \big)  \DD f -   \big(\rr_s\phi'(\rr_s) +\psi(u_s)  \rr_V\big) \<\nn V,\nn f\>  \Big)(x) \d x. \end{split}\end{equation}
\end{defn}

Since $  \rr_V/\rr_V =1$,
 $\rr_t\equiv \rr_V$ solves \eqref{E1} as a stationary distributional solution.
In this paper, we aim to construct distributional solutions and study the convergence rates of $u_t\to 1$ (equivalently, $\rr_t\to\rr_V$) in $L^2(\mu)$ and relative entropy:
$$ \lim_{t\to\infty}   \mu\big(|u_t-1|^2\big) =0,\ \ \  \lim_{t\to\infty}   \mu\big(u_t \log u_t\big) =0,$$
where, and in the following, for a measurable space $(E,\F,\nu)$,
 $$\nu(f):= \int_{E}f\d\nu,\ \ \ f\in L^1(\nu).$$

  We will need the following assumption on $V$ and the initial value $\rr_0$.

\begin{enumerate}
\item[{\bf (A)}]
$V \in C^2(\mathbb R^d)$,   $\inf V>-\infty$, $Z(V):=\int_{\R^d}\e^{-V(x)}\d x<\infty$. Moreover,  $\rr_0\in \D_1^+$ such that $u_0:=\rr_0/\rr_V$ is bounded.
  \end{enumerate}

 \beg{thm}\label{TT} Assume {\bf (A)}.
  Then the equation    $\eqref{E1}$ has a distributional solution $\rr_t$  with  initial value $\rr_0$ such that  $\|u_t\|_\infty\le \|u_0\|_\infty,\ t\ge 0.$
\end{thm}

\beg{rem}  We will prove Theorem \ref{TT} by realizing the approximation scheme proposed in the proof of \cite[Proposition 4.4]{DGGW} for   the case $\phi=0$ and $\psi(s)=c s^m$ for some constants $c>0$ and $m>1$.
  However,    unlike what is claimed therein, namely that a classical solution can be constructed, we are only able to derive   distributional solutions.
Below we recall some  known    results on the existence and uniqueness results for  distributional solutions  to \eqref{E1-1} and \eqref{E1-2},   respectively.

\begin{enumerate}
\item[$(1)$]   

  When $\phi=0$, $V\in C^\infty(\R^d; [0,\infty))$ is  convex,  with $\lim_{|x|\to \infty}V(x) =\infty$,    $\nabla V\in L^1(\rho_V)$,
$\psi\in C^2$, and $0<\psi'(r)\le M$  for some constant $M>0$, the existence of distributional solutions has been  shown in  the proof of \cite[Proposition 2.5]{R}.
  \item[$(2)$]     \eqref{E1-2}  is covered by the equation
 \begin{equation}\label{geo} \pp_t \rr_t = \DD \big[\phi(\rr_t)\big]  {+} {\rm div} \big[\rr_t b(\rr_t)\nn V\big], \end{equation}
 where $b: [0,\infty)\to \R$ is  measurable.
  According to  \cite{BR2023,BR2023JFA},  if   $b$ and $\nabla V$ are bounded, and there exists a locally bounded function $\alpha:[0,\infty)\to[0,\infty)$ such that
  $$|b(r)r-b(s)s|\leq \alpha(M)|\phi(r)-\phi(s)|,\ \ M>0,\ r,s \in [-M,M],$$
   then \eqref{geo} has a unique distributional solution in $(L^1\cap L^\infty)([0,T];\D_1^+\cap L^\infty(\R^d))$ for $T>0$
  and any bounded $\rr_0\in \D_1^+$. See also \cite{Ba} for the case with a singular convolution term.   \end{enumerate} \end{rem}

Next, we consider the convergence rate of $u_t\to 1$ (equivalently, $\rr_t\to \rr_V$) as $t\to\infty$, under the $p$-variance for $p\in [1,\infty)$:
$$\D_{1,\mu}^+\ni u\mapsto {\rm Var}_{p,\mu}(u):=\beg{cases}  \dfrac{\mu(u^p) -1}{p-1}, &\text{if}\  p\in (1,\infty),\\
\mu\big(u\log u\big), &\text{if}\  p=1.\end{cases} $$
By the dominated convergence theorem, we have
$$\lim_{p\downarrow 1}  {\rm Var}_{p,\mu}(u)= {\rm Var}_{1,\mu}(u),\ \ u\in \D_{1,\mu}^+\cap L^{1+}(\mu).  $$

To estimate the convergence rate of ${\rm Var}_{p,\mu}(u_t)\to 0$ as $t\to\infty$,
 we will apply nonlinear functional inequalities restricted on the ball
 $$B_n:=\big\{x\in \R^d:\ |x|\le n\big\}$$ for large radius $n\in\mathbb N$. In applications, these local functional inequalities   are easier to establish than the corresponding global ones. Let
\begin{align}\label{mun}\mu_n(\d x):=\ff 1 {\mu(B_n)}   1_{B_n}(x) \mu(\d x),\ \ \ n\ge 1.
\end{align}
 We make the following assumption which is satisfied if $\phi(s)$ and $\psi(s)$ are of type
 $$0\le s\mapsto \int_0^R s^\theta\nu(\d\theta)$$
 for some $R\in (0,\infty)$ and a finite measure $\nu$ on $[0,R].$

 \beg{enumerate} \item[{\bf (B)}]
 For any $M\in (0,\infty)$, there exists   $\gg: [1,\infty)\to [1,\infty)$  with  $\lim_{\theta \to 1}\gg(\theta)=1$ such that
\beq\label{GG} \sup_{s \in [\theta^{-1}, \theta]} \big(\phi'(s r),\psi'(s r)\big)\le  \gg(\theta) \big(\phi'(r),\psi'(r)\big),\ \ \ r\in (0,M], \  {\theta\geq 1}.
\end{equation}
\end{enumerate}

 \beg{thm}\label{TT2} Assume {\bf (A)} and {\bf (B)}. Then $\eqref{E1}$ has  a distributional solution with initial value $\rr_0$ such that
      the following assertions hold.
  \beg{enumerate}
 \item[$(1)$] Let $p\in [1,\infty)$. If there exists a constant $\ll>0$ such that the following $L^p$-Poincar\'e inequality holds for large $n\ge 1$:
 \beq\label{PC10} \beg{split}& {\rm Var}_{p,\mu_n} ( f)\le \ff p \ll \, {\mu_n \Big(\big[\phi'(f\rr_{V} )+\psi'(f)\big]  f^{p-2} \big|\nn  f \big|^2\Big)},\\
&\qquad  \qquad\qquad \  \ f\in C_b^1(B_n),\ \mu_n ( f)=1,\ \inf_{B_n} f>0,\end{split}\end{equation}
then
  \begin{align}\label{A2} {\rm Var}_{p,\mu} \mu(u_t)  \le \e^{-\ll t} {\rm Var}_{p,\mu}(u_0),\ \ \ t\ge 0.\end{align}
\item[$(2)$] If there exists a decreasing function $\aa: (0,\infty)\to (0,\infty)$  such that the following weak Poincar\'e  inequality holds for large $n\geq 1$:
\beq\label{A4}\beg{split}  & \mu_n  ( |f-1|^2)\le  \aa(r) \mu_n  \Big(  \big[\phi'(f\rr_{V})+\psi'(f)\big]   \big|\nn  f \big|^2\Big)+ r  \|f\|_\infty^2,\\
&\ \ \qquad  \ 0\le f\in C_b^1(B_n),\ \mu_n( f)=1, \ r>0,\end{split}  \end{equation}
then
\beq\label{A5}    \mu(|u_t-1|^2)  \le  \inf_{r>0} \Big\{\e^{-2t/\aa(r)} + r  \Big\}  \|u_0 \|_\infty^2,\ \ \ t> 0. \end{equation}
\item[$(3)$] If $\eqref{A4}$  holds  for large $n\ge 1$ and $\aa(r)= \dd (1+r^{-\theta})$   for  some constants $\dd,\theta>0$, then there exists a constant $k>0$ such that
\beq\label{A6}   \mu(|u_t-1|^2) \le k  t^{-\ff 1\theta} \|u_0  \|_\infty^2,\ \ \ t>0.\end{equation}
\end{enumerate}
 \end{thm}

In Section 2, we study nonlinear Neumann problems on $B_n$ whose solution will converge to a distributional solution of \eqref{E1} as $n\to\infty$. In Section 3 and Section 4, we prove
Theorem \ref{TT}  and Theorem \ref{TT2} respectively, which are then applied in Section 5 and Section 6 to non-degenerate and degenerate models, respectively.

 \section{Approximations of \eqref{E1}}

 For any $n\in\mathbb N$, let  $P_t^n$ be the standard Neumann semigroup on $B_n$ generated by $\DD$. Then there exists
$t_n>0$ such that
$$\tt V_{n}:= P_{t_n}^n V\in C^\infty(B_n),\ \ \sup_{B_n}|\tilde{V}_{n}-V|\le \ff 1 n.$$
Let $h\in C^\infty(\R)$ such that $0\le h\le 1, h|_{(-\infty,0]}=1,\ h|_{[1/2,\infty)}=0.$ Define
$$V_n(x):= h(|x|+1-n) V(x)+ \big(1-h(|x|+1-n)\big) \tt V_{n}(x),\ \ \ x\in B_n.$$
Then, letting $N(x):= \ff x {|x|}$ be the outside unit normal vector field of $\pp B_{n}$, we have
\beq\label{VNN}  \beg{split}&V_n\in C^2(B_n),\ \ \ V_n|_{B_{n-1}}= V|_{B_{n-1}},\\
&\|1_{B_n}(V_n-V)\|_\infty\le \ff 1 n,\ \ \  \<N, \nn V_n\>|_{\pp B_n}=0,\ \ n\ge 1.\end{split} \end{equation}

Let
$$\rr_{V_n}(x) :=\ff{\e^{-V_n(x)}}{Z_n}\ \text{for}\ x\in B_n,\ \ \ Z_n:= \int_{B_n}\e^{-V_n(x)}\d x.$$ Then
$$\hat\mu_n(\d x):= 1_{B_n(x)}\rr_{V_n}(x)\d x= \ff 1 {Z_n} 1_{B_n(x)} \e^{-V_n(x)}\d x$$
is a probability measure on $B_n$.
By  \eqref{VNN}, there exists a  sequence $\{s_n \}_{n\ge 1}  $ such that
\begin{align}\label{VRL}  1\le s_n\downarrow 1\ \text{as}\ n\uparrow\infty,\   \  s_n^{-1} \rr_V|_{B_n}\le \rr_{V_n}|_{B_n} \le s_n  \rr_V|_{B_n},\ \ n\ge 1.
\end{align}

Next, for any $n\in\mathbb N$  let $t_n'>0$ be small enough such that
\beq\label{X*} \big\|P_{t_n'}^n u_0- u_0\|_{L^2(B_n)} \le \ff 1 n.\end{equation}
Then
\beq\label{R0} u_{n,0}:=   \ff{P_{t_n'}^n u_0}{\hat\mu_n(P_{t_n'}^nu_0)},\ \ \ \
\rho_{n,0} := \rho_{V_n}  u_{n,0},\ \ \ n\ge 1 \end{equation}
satisfy \beq\label{GLN}\<\nn \rr_{n,0}, N\>|_{\pp B_n}=0,\ \ \   0\le \rho_{n,0} \in C^\infty(B_n),\ \ \ \int_{B_n}\rr_{n,0}(x)\d x= \hat\mu_n(u_{n,0})=1,\end{equation}
and combining \eqref{X*}  with {\bf (A)}, \eqref{VNN} and \eqref{VRL},   we find  a constant $M_0>0$ such that
\begin{equation}\label{intbo} \beg{split}
&\limsup_{n\to\infty} \|u_{n,0}\|_\infty\le \|u_0\|_\infty, \\
& \sup_{n\ge 1} ( \|\rr_{n,0} \|_{\infty} +\|u_{n,0}\|_\infty +\|\rr_{V_n}\|_\infty\big)
  \le M_0,\\
  & \sup_{n\ge 1}  \rho_{n,0} 1_{B_n} \le  M_0  \rr_V  \in L^1(\R^d),\\
  &  \lim_{n\to\infty}  (\rr_{n,0},u_{n,0}) (x)= (\rr_0,u_0)(x),\ \ \text{in }  \ L^2_{loc}(\R^d). \end{split} \end{equation}

By $\phi,\psi \in C^2([0,\infty))$   with $\phi',\psi'\ge 0$, we may choose   $\phi_n,\psi_n\in C_b^\infty(\R)$ such that
\beq\label{PPSN} \phi_n',\psi_n'\ge 0,\ \ \   \lim_{n\to\infty}   \big( \|\phi_n - \phi \|_{C_b^2([0,M_0]) }  + \|\psi_n - \psi \|_{C_b^2([0,M_0]) }   \big)=0.
  \end{equation}
 Let
\beq\label{AN} A_n(\rr) := \phi_n'(\rr) + \psi_n'\left(\rr /\rho_{V_n}\right)+n^{-1},\ \  0\le \rr \in \B(B_n).
\end{equation}
 We  consider the following nonlinear Neumann problem  on $B_n$ with initial value $\rr_{n,0}$:
  \begin{equation}\label{rvr}\beg{split}
&\partial_t \rho_{n,t}= \Div\big\{A_{n}(\rho_{n,t}) \rr_{V_n} \nn u_{n,t}  \big\},\ \ \ u_{n,t}:= \ff{\rr_{n,t}}{\rr_{V_n}},\\
&\qquad \ \ \ \<\nabla\rho_{n,t}, N\>|_{\pp B_n}=0,\ \ \ t\ge 0.\end{split}\end{equation}

 \beg{prp} Assume {\bf (A)}. Then for any $n\ge 1$, the equation $\eqref{rvr}$ has a unique solution in $C^{2,1}(B_n\times [0,\infty)).$
 \end{prp}

 \beg{proof} We first recall a result in \cite{LSU}, then apply to the present setting.

 (a)  Let $\OO\subset \R^d$ be a bounded  smooth open domain, let $T\in (0,\infty)$, and let
$$Q_T:= \OO\times (0,T),\ \ \ S_T:= (\pp\OO)\times [0,T].$$
Let $N$ be the outward unit normal vector field on $\pp \OO$, and for $w\in C^{2,1}(\bar\OO\times [0,T]),$ let
\beg{align*} &\scr L w_t(x):= \pp_t w_t(x)-\sum_{i,j=1}^d a_{ij}(x, w_t(x))\pp_i\pp_j w_t(x) + b(x,w_t(x),\nn w_t(x)),\ \ (x,t)\in Q_T,\\
&\scr L^{(S)} w_t(x) = \big\<a (x, w_t(x)) \nn w_t(x), N(x)\big\>,\ \ \ (x,t)\in S_T,\end{align*}
where
$$a:=(a_{ij})_{1\le i,j\le d}:\ \bar\OO\times\R\to \R^d\otimes\R^d,\ \  b: \bar\OO\times\R\times\R^d\to \R$$
are measurable.
Then, according to (7.1) and (7.2) in  \cite[Chapter V]{LSU} with $\psi=0$, the equation (7.40) in \cite[Chapter V]{LSU} for $\tau=1$, where $``=0$" is missed in the first line,   becomes
\beq\label{L*} \beg{split} &\pp_t w_t (x) =\sum_{i,j=1}^da_{ij} (x, w_t(x)) \pp_i\pp_j w_t(x) - b(x, w_t(x), \nn w_t(x)),\ \ \ (x,t)\in Q_T,\\
&\big\<a (x, w_t(x)) \nn w_t(x), N(x)\big\> =0,\ \ \  (x,t)\in S_T,\\
&w_0(x)=0,\ \ \ x\in \bar\OO.\end{split}\end{equation}
According to  Theorem 7.4  in  \cite[Chapter V]{LSU} with condition (7.34) in (a) for $\psi=0$ therein,  this equation has a unique solution in $C^{2,1}(\bar Q_T)$ provided the following conditions hold (recall
that we set $\psi=0$ therein) for some  constants  $c_1>c_2> 0$ and all $R\in (0,\infty)$:
  \begin{align*}&c_2 |\xi|^2 \le  \big\<a (x,r)\xi,\ \xi\big\>\le c_1|\xi|^2, \ \ x\in \bar{\Omega}, r\in \R,\ \xi\in\R^d,\\
&|b(x,r,\eta)|\leq c_1(1+ |r|+|\eta|^2),\ \ \    x \in \bar{\OO},\ r\in\R,\ \eta\in\R^d,\\
&\sup_{(x,r)\in\bar\OO\times [0,R]} \Big( \| \partial_r a (x,r) \|+\|\nn_x a (x,r)\|
 + \|\nn_x^2 a  (x,r)(x) \| +\|\nn_x  \pp_r a (x,r)\|  \Big)<\infty, \\
&\sup_{(x,r,\eta)\in \bar\OO\times [0,R]\times\R^d} \bigg(\ff{|\nn b(x,r,\cdot)(\eta)| }{1+|\eta|}+\ff{|\pp_rb (x,r,\eta)|}{1+|\eta|^2} \bigg)<\infty, \end{align*}
 and moreover $b(x,r,\eta) $ is H\"older continuous in $x\in \bar\OO$ uniformly in $|r|, |\eta|\le R$.

(b) To apply the above result to the present setting, let  $\OO=B_n\setminus \pp B_n,$  and  let
$$ w_t:= \rr_{n,t}-\rr_{n,0}, \ \ t\ge 0.$$
By \eqref{GLN}, $w_t$ satisfies the Neumann boundary condition if and only if so does $\rr_t$.
Note that
\beq\label{cruci} \nn u_{n,t}= \rr_{V_n}^{-1}\big(\nn\rr_{n,t}+ \rr_{n,t}\nn V_n\big),\ \ t\ge 0,\ n\ge 1.\end{equation}
By straightforward calculations,     \eqref{rvr} is formulated as  \eqref{L*} with
\beg{align*} &a_{ij}(x,r)= \dd_{ij} I_{1,n}(x,r),\ \ \ 1\le i,j\le d,\\
& b(x,r,\eta)=I_{1,n}(x,r) \Big(\big(r+\rr_{n,0}(x)\big) \DD V_n(x) + \big\<\eta+\nn \rr_{n,0}(x), \nn V_n(x)\big\>\Big)\\
&\quad  +\big\<I_{2,n}(x,r,\eta)+ I_{3,n}(x,r,\eta), \ \eta+\nn \rr_{n,0}(x)+\big(r+\rr_{n,0}(x)\big)\nn V_n(x)\big\>,\end{align*}
where
\beg{align*}
& I_{1,n}(x,r):=n^{-1}+\phi_n'\big(r+\rr_{n,0}(x)\big) + \psi_n'\big((r+\rr_{n,0}(x))/\rr_{V_n}(x)\big),\\
& I_{2,n} (x,r,\eta):= \phi_n''\big((r+\rr_{n,0}(x)\big) \big(\eta+\nn \rr_{n,0}(x)\big),\\
& I_{3,n} (x,r,\eta):= \psi_n''\bigg(\ff{r+\rr_{n,0}(x)}{\rr_{V_n}(x)}\bigg)\rr_{V_n}(x)^{-1} \Big( \eta+\nn \rr_{n,0}(x) +\big(r+\eta_{n,0}(x)\big)\nn V_n(x)\Big).\end{align*}
 Since  $0\le \phi'_n, \psi_n' \in C_b^\infty(\R)$ and   $\rr_{n,0}, V_n\in C^2 (B_n)$, the above conditions hold. Then  the proof is finished.
 \end{proof}

  By \eqref{VNN} and the boundary condition in \eqref{rvr}, we have
\beq\label{BD} \<\nabla \rho_{n,t},N\>|_{\pp B_n}= \<\nabla V_n,N\>|_{\pp B_n}=0,\ \ t>0.\end{equation}
Let
 \begin{equation}\label{9}
E_n(t)   :=\hat\mu_n\big(u_{n,t}\log u_{n,t}\big) =\int_{B_n} \left(u_{n,t} \log u_{n,t}    \right) \d \hat\mu_n,\ \ t\ge 0.\end{equation}

\begin{thm}\label{pri} Assume {\bf (A)}. Then the following assertions hold.
\begin{enumerate}
\item[$(1)$] Non-negativity and uniform boundedness:  $0\le u_{n,t}  \le \|u_{n,0}\|_\infty$,  and
\beq\label{LX}  \sup_{n\ge 1, t\ge 0} \big(\|u_{n,t}\|_\infty+\|\rr_{n,t}\|_\infty+\|\rr_{V_n}\|_\infty\big)  \le M_0.\end{equation}
\item[$(2)$] Mass conservation: $\int_{B_n}\rho_{n,t}(x)\d x=\int_{B_n}u_{n,t}\d \hat\mu_n=1,\ t\ge 0,\ n\ge 1.$
\item[$(3)$] Energy estimate: $E_n(t)\ge 0$ and
\begin{align}\label{goo}   \sup_{n\ge 1,t\ge 0} E_n(t) +\sup_{n\ge 1}\int_0^\infty  \d s \int_{B_n}\Big(  A_{n}(\rho_{n,s})\frac{|\nabla u_{n,s} |^2}{u_{n,s}} \Big) \d  \hat\mu_n
<\infty.
\end{align}
 \end{enumerate}
\end{thm}

\begin{proof}(1)  We first prove $u_{n,t}\ge 0$.
 Let $u_{n,t}^{-}:=\max\{0, - u_{n,t}\}$. We have
$$\pp_t \big(u_{n,t}^-\big)^2  = - 2u_{n,t}^-  \ff{\pp_t \rr_{n,t}}{\rr_{V_n}}=-2u_{n,t} ^{-}\left[\Div(A_n(\rho_{n,t})\nabla \rho_{n,t}) + \Div(A_n(\rho_{n,t})\rho_{n,t}\nabla V_n)\right]$$ and a.e.
$$\nn u_{n,t}^{-} = - 1_{\{u_{n,t} <0\}} \rr_{V_n}^{-1} \big(\nn\rr_{n,t} + \rr_{n,t} \nn V_n\big).  $$
By combining these with   the dominated convergence theorem, the integration by parts formula, \eqref{cruci},    \eqref{BD} and $A_n(\rr_{n,t})\ge 0$,  we derive
 \begin{align*}
& \frac{1}{2}\partial_t\int_{B_n}\Big(\left(u_{n,t}^{-}\right)^2\rho_{V_n}\Big)(x) \d x=  \frac{1}{2} \int_{B_n}\Big(\pp_t \left(u_{n,t}^{-}\right)^2\rho_{V_n}\Big)(x) \d x\\
&=   -  \int_{B_n}\Big(u_{n,t} ^{-} \Div\big\{A_n(\rho_{n,t})(\nabla \rho_{n,t}  +  \rho_{n,t}\nabla V_n)\big\} \Big)(x) \d x  \\
&=   \int_{B_n}A_n(\rho_{n,t})\left\<\nabla u_{n,t}^{-}, \nabla \rho_{n,t}+\rho_{n,t}\nabla V_n)\right\>\d x  \\
&\quad  -  \int_{\partial B_n} u_{n,t}^{-}A_n(\rho_{n,t})\big\< \nabla \rho_{n,t}+ \rho_{n,t}\nabla V_n, N\big\>\d S \\
&= -\int_{B_n}\Big(A_n(\rho_{n,t})\left|\nabla u_{n,t}^- \right|^2\rho_{V_n}\Big)(x)\d x\leq 0.
\end{align*} Since $u_{n,0}^{-}=0$, we obtain  $u_{n,t}^{-}=0$, so that $u_{n,t}\ge 0$ for $t>0$.

Similarly, by \eqref{rvr},  the dominated convergence, the integration by parts formula, \eqref{BD}  and   \eqref{cruci}, for any $p>1$ we have
\beq\label{ERG}   \begin{split}
&\partial_t \int_{B_n}\big(u_{n,t}^p\rho_{V_n}\big)(x)\d x\\
&= p\int_{B_n}\Big(\left(u_{n,t}\right)^{p-1}\left(\Div(A_n(\rho_{n,t})\nabla \rho_{n,t}) + \Div(A_n(\rho_{n,t})\rho_{n,t}\nabla V_n)\right)\Big)(x)\d x\\
&= -p(p-1)\int_{B_n}\Big(u_{n,t}^{p-2}A_n(\rho_{n,t})\left\<\nabla u_{n,t}, \nabla \rho_{n,t}+\rho_{n,t}\nabla V_n)\right\>\Big)(x) \d x\\
&+p\int_{\partial B_n}u_{n,t}^{p-1}A_n(\rho_{n,t})\big\< \nabla \rho_{n,t}+ \rho_{n,t}\nabla V_n, N\big\>\d S\\
&= -p(p-1)\int_{B_n}\Big(u_{n,t}^{p-2}A_n(\rho_{n,t})\left|\nabla u_{n,t} \right|^2\rho_{V_n}\Big)(x) \d x\leq 0.
\end{split} \end{equation}
This means $$\int_{B_n}u_{n,t}^p\rho_{V_n}\d x\leq \int_{B_n}u_{n,0}^p\rho_{V_n}\d x,\ \ \ p>1.$$
Consequently,  $\|u_{n,t} \|_{\infty}\leq \|u_{n,0} \|_\infty$, which implies    \eqref{LX} by  \eqref{intbo}.

(2) Similarly to (1), by the integration by parts formula, \eqref{BD} and \eqref{cruci}, we obtain
\begin{align*}
&\partial_t \int_{B_n}\rho_{n,t}(x)\d x\\
&= \int_{B_n}\left(\Div(A_n(\rho_{n,t})\nabla \rho_{n,t}) + \Div(A_n(\rho_{n,t})\rho_{n,t}\nabla V_n)\right)(x)\d x=0.
\end{align*}  This and \eqref{GLN} implies
$$\int_{B_n} \rr_{n,t}(x)\d x=\int_{B_n} u_{n,t}\d\hat\mu_n=1,\ \ \ t\ge 0.$$

(3)  By  $0\le u_{t,n}\le \|u_0\|_\infty$, $\hat\mu_n(u_{n,t})=1$  and that $\hat\mu_n $ is a probability measure on $B_n$, we obtain $E_n(t)\ge 0$ and
\beq\label{LQ} \sup_{n\ge 1, t\ge 0} E_n(t) \le \max\big\{\e,  \|u_0\|_\infty\log \|u_0\|_\infty\big\}.\end{equation}
Next, for any $\vv>0$, let
$$E_{n,\vv}(t) := \int_{B_n} \left(u_{n,t} \log(\vv+ u_{n,t} )   \right)(x)\rho_{V_n}(x) \d x.$$
By  \eqref{rvr},  the integrating by parts formula, \eqref{BD} and \eqref{cruci}, we derive
\beq\label{EN*} \begin{split}
\frac{\d E_{n,\vv}(t)}{\d t}&=\int_{B_n} \Big(\big[\log(\vv+ u_{n,t})\big] \Div\big\{A_n(\rho_{n,t})\big(\nabla \rho_{n,t}+  \rho_{n,t}\nabla V_n\big)\big\}  \Big)(x)\d x\\
& {+\int_{B_n} \left(\left[\frac{u_{n,t}}{\vv+ u_{n,t}}\right] \Div\big\{A_n(\rho_{n,t})\big(\nabla \rho_{n,t}+  \rho_{n,t}\nabla V_n\big)\big\}  \right)(x)\d x}\\
&=-\int_{B_n} \left(\frac{A_n(\rho_{n,t})}{\vv+u_{n,t}}  \left\<\nabla u_{n,t}, \nabla \rho_{n,t}+\rho_{n,t}\nabla V_n\right\>\right)(x)\d x\\
& {-\int_{B_n} \left(\vv\frac{A_n(\rho_{n,t})}{(\vv+u_{n,t})^2}  \left\<\nabla u_{n,t}, \nabla \rho_{n,t}+\rho_{n,t}\nabla V_n\right\>\right)(x)\d x}\\
&\leq -\int_{B_n}  \Big(A_n(\rho_{n,t})\frac{|\nabla u_{n,t}|^2}{\vv+u_{n,t}}\Big) \d \hat\mu_n.\end{split}\end{equation}
So,
 {$$\int_0^t\int_{B_n} \Big(A_n(\rho_{n,s})\frac{|\nabla u_{n,s}|^2}{\vv+u_{n,s}} \Big) \d \hat\mu_n\d s \le E_{n,\vv}(0)- E_{n,\vv}(t).$$}
Letting $\vv\to 0$ and  combining with \eqref{LQ}, we
 derive  \eqref{goo}.
\end{proof}

\section{Proof of Theorem \ref{TT}}

For any $p\in (0,\infty)$,  let $H^p(B_n)$ be the closure of functions $f\in C_0^\infty(B_n\setminus \pp B_n)$   under the Sobolev norm
$$\|f\|_{H^p(B_n)}:= \|(-\DD)^{\ff p 2}f\|_{L^2(B_n)},$$
where $\DD$ is the Dirichlet Laplacian on $B_n$. By the Sobolev embedding theorem, for $p\in (\ff d 2 \lor 1,\infty)$ there exists a constant $C_{n,p}\in (0,\infty)$ such that
\beq\label{SB*} \|\nn f\|_{L^2(B_n)} + \|f\|_{L^\infty (B_n)}\le C_{n,p} \|f\|_{H^p(B_n)},\ \ \ f\in H^{p}(B_n).\end{equation}
Let $(H^{-p}(B_n), \|\cdot\|_{H^{-p}(B_n)})$ be the dual space of $(H^p(B_n), \|\cdot\|_{H^p(B_n)})$ with respect to $L^2(B_n).$

Moreover, let $W^{1,2} (B_n)$ be the closure of $C_b^\infty(B_n)$ under the Sobolev norm
$$\|f\|_{W^{1,2}(B_n)}:= \|f\|_{L^2(B_n)} + \|\nn f\|_{L^2(B_n)},$$
and let $W_{loc}^{1,2}(\R^d)$ be the class of functions $f$ on $\R^d$ such that $f|_{B_n}\in W^{1,2}(B_n)$ for any $n\ge 1.$

 To show that $\rr_{n_k,t}\to \rr_t$ weakly in $L^2_{loc}(\R^d)$ as $n_k\to\infty$ for a subsequence $n_k$, and the limit is a distributional solution of \eqref{E1},  we present some lemmas. Let
\beg{align*}
 &Z_{n,t} := \psi_n(u_{n,t}), \qquad \ \ Y_{n,t} := \phi_n(\rho_{n,t}),\ \ \ n\ge 1,\ t\ge 0.
\end{align*}

\beg{lem}\label{LK1} Assume {\bf (A)}.There exists  $K: \mathbb N\to (0,\infty)$ such that
 {\beq\label{GB1}\begin{split} &\sup_{t\ge 0, n\ge m}\Big(\int_{B_m}\big(Z_{n,t}^2+Y_{n,t}^2\big)\d x \\
&  \qquad +\ff 1 {1+t} \int_0^t \d s \int_{B_m} \big(|\nabla Z_{n,s}|^2 + |\nabla Y_{n,s} |^2\big)(x)\d x \Big) \le K_m ,\ \    m\ge 1,\end{split}\end{equation}}
\beq\label{GB*}  {\sup_{ n\ge m} \int_0^\infty\big(\|\pp_s\rr_{n,s}\|^2_{H^{-1}(B_m)}+\|\pp_su_{n,s}\|^2_{H^{-1}(B_m)}\big)\d s\le K_m,\ \ \  m\ge 1.}
\end{equation}
\end{lem}
\beg{proof}  By \eqref{LX} and
 \eqref{PPSN}, we obtain
 \beq\label{KV6}  \sup_{n\ge 1,t\ge 0}\big(\|Z_{n,t}\|_\infty+ \|Y_{n,t}\|_\infty\big)<\infty.\end{equation}
By  \eqref{PPSN}, \eqref{AN} and \eqref{LX}, we find a constant $c_1>0$ such that
\beq\label{GB0} A_n(\rr_{n,t}) \ff{|\nn u_{n,t}|^2}{u_{n,t}}\ge \psi_n'(u_{n,t}) \ff{|\nn u_{n,t}|^2}{u_{n,t}}= \ff{|\nn\psi_n(u_{n,t})|^2}{u_{n,t}\psi'_n(u_{n,t})}\ge c_1 |\nn\psi_n(u_{n,t})|^2.\end{equation}
Moreover, by  \eqref{AN} and \eqref{cruci},
\beg{align*} &A_n(\rr_{n,t}) \ff{|\nn u_{n,t}|^2}{u_{n,t}}\ge \phi_n'(\rr_{n,t}) \ff{|\nn\rr_{n,t}+ \rr_{n,t}\nn V_n|^2}{\rr_{V_n}\rr_{n,t}}\\
&\ge  \ff{ \phi_n'(\rr_{n,t}) } {\rr_{V_n}\rr_{n,t}}   \bigg( \ff 12 |\nn\rr_{n,t}|^2  -   3|\rr_{n,t}\nn V_n|^2 \bigg)\\
&=    \ff { |\nn \phi_n(\rr_{n,t})|^2} {2 \phi_n'(\rr_{n,t}) \rr_{V_n}\rr_{n,t}} -   \ff{3 \phi_n'(\rr_{n,t})\rr_{n,t}|\nn V_n|^2}{\rr_{V_n} }.\end{align*}
Combining this with \eqref{PPSN} and \eqref{LX},  and
 \beq\label{RMM} \sup_{n\ge m}   \|1_{B_m} (|V_n|+|\nn V_n|)\|_\infty<\infty\end{equation} due to \eqref{VNN},    we find a constant $c_2>0$ and a map $k: \mathbb N\to (0,\infty)$ such that
$$|\nn \phi_n(\rr_{n,t})|^2\big|_{B_m}  \le  c_2  A_n(\rr_{n,t}) \ff{|\nn u_{n,t}|^2}{u_{n,t}}\bigg|_{B_m} + k_m, \ \ n\ge m\ge 1.$$
Combining this with \eqref{GB0}, we  find a constant $c_3>0$ such that
  \beg{align*} &\big(|\nn Z_{n,t}|^2+|\nn Y_{n,t}|^2\big)\big|_{B_m}=  \big(|  \nn \psi_n(u_{n,t})|^2+     |   \nn \phi_n(\rr_{n,t})|^2 \big)\big|_{B_m}\\
& \le c_3   A_n(\rr_{n,t}) \ff{|\nn u_{n,t}|^2}{u_{n,t}}\bigg|_{B_m} +k_m,\ \ \ n\ge m\ge 1,\ t\ge 0.\end{align*}
Combining this with      \eqref{goo} and \eqref{KV6},   we find an increasing function $K: \mathbb N\to (0,\infty)$ such that \eqref{GB1} holds.

Next, by \eqref{rvr}, \eqref{cruci}   and \eqref{LX}, we find a constant $c(m)>0$ such that for any $f\in H^1(B_m)$ and $n\ge m$, we have
\beg{align*} &\bigg|\int_{B_m} \big(f\pp_t\rr_{n,t} \big)(x)\d x\bigg| + \bigg|\int_{B_m} \big(f\pp_tu_{n,t} \big)(x)\d x\bigg| \\
&= \bigg|\int_{B_m}\Big(\rr_{V_n} A_n(\rr_{n,t}) \big\<\nn f,  \nn u_{n,t}\big\>\Big)(x)\d x\bigg|\\
&\quad+
\bigg|\int_{B_m}\Big(\rr_{V_n} A_n(\rr_{n,t}) \big\<\nn(\rr_{V_n}^{-1} f),  \nn u_{n,t}\big\>\Big)(x)\d x\bigg|\\
&\le   c(m) \|f\|_{H^1(B_m)} \bigg(\int_{B_m} \Big(
A_n(\rr_{n,t}) \ff{|\nn u_{n,t}|^2}{u_{n,t}} \Big)(x)\d x\bigg)^{\ff 1 2},\end{align*}
Combining this with      \eqref{goo},   we find an increasing function $K: \mathbb N\to (0,\infty)$ such that \eqref{GB*} holds.
\end{proof}

\beg{lem}\label{LK2} Assume {\bf (A)}  and let $p\in (\ff d 2\lor 1,\infty)$. Then  there exists   $K: \mathbb N\to (0,\infty)$ such that
   \beq\label{GB2}     \sup_{t\ge 0, n\ge m} \frac{1}{1+t} \int_0^t\big(\|\pp_s Y_{n,s}\|_{H^{-p}(B_m)} +  \|\pp_s Z_{n,s}\|_{H^{-p}(B_m)}\big)\d s \le K_m,\ \ \ m\ge 1. \end{equation}
   \end{lem}

   \beg{proof}
 By  \eqref{rvr} and \eqref{cruci}, we have
 $$\partial_t Z_{n,t}= \psi_n'(u_{n,t})  \frac{\partial_t \rho_{n,t}}{\rho_{V_n}}= \ff{ \psi_n'(u_{n,t})}{\rr_{V_n}} {\rm div}\Big[ A_n(\rr_{n,t}) \rr_{V_n} \nn u_{n,t}\Big].$$
 So, for any $f\in H^p(B_m)$,
 \beg{align*}
& \bigg|  \int_{B_m} \big(f  \partial_t Z_{n,t}\big)(x)\d x\bigg|
 = \bigg|  \int_{B_m} \Big(A_n(\rr_{n,t}) \Big\<\rr_{V_n} \nn u_{n,t},\ \nn \big\{f \rr_{V_n}^{-1} {\psi_n'(u_{n,t}) } \big\}\Big\>   \Big)(x)\d x\bigg| \\
&\le  \int_{B_m}  \Big(A_n(\rho_{n,t})|\nn u_{n,t}|  |\nn f|  {\psi_n'(u_{n,t}) }\Big)(x) \d x\\
&\quad  + \int_{B_m}  \Big[A_n(\rho_{n,t})|f \nn u_{n,t}|   \Big(\big|\nn {\psi_n'(u_{n,t}) }\big|  +   {\psi_n'(u_{n,t}) } |\nn V_n| \Big) \Big](x)  \,\d x\\
&\le  \|\nn f\|_{L^2(B_m)} \bigg(\int_{B_m} \Big(A_n(\rho_{n,t}) \psi_n'(u_{n,t})  |\nn u_{n,t}| \Big)^2(x)\d x\bigg)^{\ff 1 2}\\
&\quad + \|f\|_{L^\infty(B_m)}  \int_{B_m}  \Big[A_n(\rho_{n,t})| \nn u_{n,t}|^2  |\psi_n''(u_{n,t})|    +   {\psi_n'(u_{n,t}) } |\nn V_n| \Big) \Big](x)  \,dx.\end{align*}
By \eqref{VNN}, \eqref{PPSN}  and \eqref{LX}, we find a constant $k_m  \in (0,\infty)$ such that
\beg{align*} &\Big(A_n(\rho_{n,t}) \psi_n'(u_{n,t})  |\nn u_{n,t}| \Big)^2 +A_n(\rho_{n,t})| \nn u_{n,t}|^2  |\psi_n''(u_{n,t})| \le k_m A_n(\rho_{n,t})\rr_{V_n} \ff{| \nn u_{n,t}|^2}{u_{n,t}},\\
& \psi_n'(u_{n,t})  |\nn V_n| \le k_m,\ \ \ t\ge 0,\ n\ge m\end{align*}
hold on $B_m$.
Combining these with   \eqref{goo} and \eqref{SB*},  we   obtain
$$\sup_{t\ge 0,n\ge 1} \|\pp_t Z_{n,t}\|_{H^{-p}(B_m)}<\infty.$$
Similarly, we can prove the same property for $ \|\pp_t Y_{n,t}\|_{H^{-p}(B_m)}$.   So,   \eqref{GB2} holds for some constant $K_m>0.$
  \end{proof}

 \beg{proof}[Proof of Theorem \ref{TT}]   We first construct $\rr$, then verify conditions  in Definition \ref{D1.2}.

 {\bf  (1)}  Noting that $W^{1,2}(B_n)$ is compactly embedded into $L^2(B_n)$, and  $L^2(B_n)$ is continuously embedded into $H^{-p}(B_n)$,
 by    the Aubin-Lions lemma,
     \eqref{GB1} and \eqref{GB2} imply that the sequence $(Z_{l,\cdot}, Y_{l,\cdot})_{l\ge n}$ is relatively compact in $L^2_{loc}([0,\infty)\times B_n).$
 So,     for any $n\ge 1$, we find a  sequence $n\le n_k\uparrow \infty$ as $k\uparrow\infty$, and
$Z^{(n)},Y^{(n)} \in L_{loc}^2([0,\infty)\times B_n)$, such that
$$ \lim_{k\to\infty}  (Z_{n_k,\cdot}, Y_{n_k,\cdot} ) = (Z^{(n)},Y^{(n)}) \quad \text{  in\  } L_{loc}^2([0,\infty)\times B_n).$$
 Then, up to a subsequence,
 \beq\label{GO0}  \lim_{k\to\infty}  (Z_{n_k,\cdot}, Y_{n_k,\cdot} ) = (Z^{(n)},Y^{(n)}) \quad \text{ a.e.  \ on\ } [0,\infty)\times B_n.\end{equation}

By  \eqref{LX} and \eqref{GB*},
   $\{\rr_{n_k,\cdot}, u_{n_k,\cdot}\}_{k\ge 1}\subset C([0,\infty); H^{-1}(B_n))$ are equicontinuous.
Since $  L^2(B_n)$    is compactly embedded into $H^{-1}(B_n)$,
by \eqref{LX} and the Arzel\'a-Ascoli theorem,
there exist   $\rr^{(n)}, u^{(n)}\in C([0,\infty); H^{-1}(B_n))\cap \B_b([0,\infty)\times\R^d)$ such that, up to a subsequence,
$$\lim_{k\to\infty} \sup_{t\in [0,n]}   \big(\|\rr_{n_k,t}- \rr_t^{(n)}\|_{H^{-1}(B_n)}+  \|u_{n_k,t}- u_t^{(n)}\|_{H^{-1}(B_n)}\big)=0.$$
Combining this   with \eqref{LX}     we obtain
\beq\label{GO2}  \lim_{k\to\infty} (\rr_{n_k,t},  u_{n_k,t})= (\rr_t^{(n)}, u_t^{(n)}) \quad \text{ weakly\  in }\  L^2(B_n),\ \  \ t\ge 0.\end{equation}
Since   $(1_{B_n}\rr_{n,0},1_{B_n}u_{n,0})\to (\rr_0,u_0)$ in $L^2_{loc}(\R^d)$ due to \eqref{intbo}, this implies
\beq\label{X**} (\rr^{(n)}_0, u^{(n)}_0) = (\rr_0,u_0)|_{B_n}.\end{equation}
Moreover,
since  $u_{n_k,t}=\rr_{n_k,t}/\rr_{V_{n_k}}$,  by \eqref{VRL} and  \eqref{GO2} we obtain
\beq\label{GO3} u_t^{(n)}=\rr_t^{(n)}/ \rr_{V},\ \ \ t\ge 0, \ n\ge 1,\end{equation}
while Theorem \ref{pri}(1) and \eqref{GO2}   imply
\beq\label{GO4}  \sup_{t\ge 0,n\ge 1} \|u_{n,t}\|_\infty \le \|u_0\|_\infty,\ \ \  \sup_{t\ge 0,n\ge 1} \|\rr_{t,n}\|_\infty \le \|u_0\|_\infty\|\rr_V\|_\infty.   \end{equation}

 By taking $\{(n+1)_k\}_{k\ge 1}$ as a subsequence of $\{n_k\}_{k\ge 1}$ for given $n\ge 1$, we have
 $$(Z^{(n)},Y^{(n)})|_{[0,\infty)\times B_n}=(Z^{(n+1)},Y^{(n+1)})|_{[0,\infty)\times B_n}\ \text{a.e.\ on} \ [0,\infty)\times B_n,$$
 and $\rr_t^{(n)} = \rr_t^{(n+1)}|_{B_n}$ for  $  t\ge 0$. So,
  $$(Y,Z,\rr,u) :=\sum_{n=1}^\infty 1_{B_n\setminus B_{n-1}} (Z^{(n)},Y^{(n)}, \rr^{(n)}, u^{(n)}),\ \  B_0:=\emptyset$$
  satisfies
\beg{align*} &(Y,Z)|_{B_n} =  (Y^{(n)}, Z^{(n)})\ \ \text{a.e.\ on} \ [0,\infty)\times B_n,\ n\ge 1,\\
  &(\rr_t,u_t)|_{B_n}= (\rr_t^{(n)}, u_t^{(n)}),\ \ \text{a.e.\ on} \  B_n,\ n\ge 1,\ t\ge 0,\end{align*}
  where the initial value  of $(\rr, u)$ coincides with that in \eqref{E1}.

 By \eqref{GO0}-\eqref{GO4} and    the diagonal principle, we
find a common subsequence, denoted again  by $\{n_k\}_{k\ge 1} $,   such that
 \beq\label{z00}  \lim_{k\to\infty}  (1_{B_{n_k}}Z_{n_k,\cdot}, 1_{B_{n_k}}Y_{n_k,\cdot} ) = (Z,Y) \quad \text{ a.e.  \ on\ } [0,\infty)\times\R^d.\end{equation}
\beq\label{GB3}   \lim_{k\to\infty} \big(1_{B_{n_k}}\rr_{n_k,t}, 1_{B_{n_k}}u_{n_k,t}) = (\rr_t, u_t)\  \quad \text{weakly in }\  L^2_{loc}(\R^d),\ \  \ t\ge 0,\end{equation}
 \beq\label{GB4}   u=\ff{\rr}{\rr_V}, \ \ \sup_{t\ge 0} \|u_t\|_\infty \le \|u_0\|_\infty,\ \ \sup_{t\ge 0}\|\rr_t\|_\infty\le \|u_0\|_\infty \|\rr_V\|_\infty.\end{equation}

{\bf (2)}   We intend to prove that $\rr_t\in \D_1^+$ for $t\ge 0$, which follows from $\rr_{n_k,t}\in\D_1^+$ due to Theorem \ref{pri}, and the following formula:
\beq\label{W*} \lim_{k\to\infty} \int_{B_{n_k}} (f\rr_{n_k,t})(x)\d x =   \int_{\R^d} (f\rr_t)(x)\d x,\ \ \ f\in \B_b(\R^d).\end{equation}
By \eqref{VRL}, \eqref{GB4} and
 $$u_{n_k,t}= \rr_{n_k, t}/\rr_{V_{n_k}}, \ \ \hat\mu_{n_k}= (1_{B_{n_k}} \rr_{V_{n_k}})(x)\d x,\ \ \mu(\d x)= \rr_V(x)\d x,$$
\eqref{W*} is equivalent to
 \beq\label{W**} \lim_{k\to\infty} \int_{B_{n_k}} (fu_{n_k,t}) \d\hat\mu_{n_k}  =   \int_{\R^d} (fu_t) \d \mu,\ \ \ f\in \B_b(\R^d).\end{equation}
It suffices to prove \eqref{W*} for $0\le f\in \B_b(\R^d)$.

Let $0\le f\in \B_b(\R^d)$. By Theorem \ref{pri}(1) and \eqref{GB3}, we have $f\rr_t\ge 0$ and
\beg{align*}& \int_{\R^d}(f\rr_t)(x)\d x=\lim_{m\to\infty} \int_{B_m} (f\rr_t)(x)\d x = \lim_{m\to\infty}  \lim_{k\to\infty}\int_{B_m}(f\rr_{n_k,t})(x)\d x\\
&  \le   \liminf_{k\to\infty}\int_{B_{n_k}}(f\rr_{n_k,t})(x)\d x. \end{align*}
   On the other hand, by Theorem \ref{pri}(1) and \eqref{VRL}, we find a constant $c>0$ such that
   $$\rr_{n_k,t}= \ u_{n_k,t} \rr_{V_{n_k}}\le c \rr_V\in L^1(\R^d),\ \ t\ge 0,\ k\ge 1.$$
    So,  by Theorem \ref{pri}(2)  and \eqref{GB3}, for any $m\ge 1$ we have
  \beg{align*} &\int_{B_m}(f\rr_t)(x)\d x =\lim_{k\to\infty}\int_{B_m} (f\rr_{n_k,t})(x)\d x \\
  &\ge \limsup_{k\to\infty} \int_{B_{n_k}} (f\rr_{n_k,t})(x)\d x - c \|f\|_\infty \int_{B_m^c}\rr_V(x)\d x.
   \end{align*}
  Letting $m\to\infty$ we derive
  $$ \int_{\R^d}(f\rr_t)(x)\d x  \ge \limsup_{k\to\infty} \int_{B_{n_k}} (f\rr_{n_k,t})(x)\d x.$$
  Therefore, \eqref{W*} holds.

{\bf (3) }     We claim that
\beq\label{GB40}  \lim_{k\to\infty} (1_{B_{n_k}}\rr_{n_k,t}(x), 1_{B_{n_k}}u_{n_k,t}(x)) =(\rr_t(x), u_t(x)),  \ \ \text{a.e.} \ (t,x) \in [0,\infty)\times \R^d.\end{equation}
 Without loss of generality, we assume that $\phi'|_{(0,\infty)}>0$ so that   $\phi$ is continuous and strictly increasing on $[0,\infty)$.

By  $ Y_{n,t} = \phi(\rr_{n,t}), $    \eqref{GO0} and \eqref{PPSN},  we obtain that   up to a subsequence,
$$ Y_t  (x)= \lim_{k\to\infty}   \phi_{n_k}(\rr_{n_k,t}) (x)= \lim_{k\to\infty}   \phi (\rr_{n_k,t}) (x),\ \ \text{a.e.}\ (t,x)\in [0,\infty)\times\R^d.$$
Since $\phi$ is continuous and strictly increasing, this  is equivalent to
$$ \phi^{-1}(Y_t(x) )   = \lim_{k\to\infty}   \rr_{n_k,t}  (x),\ \ \text{a.e.}\ (t,x)\in [0,\infty)\times\R^d,$$
which together with \eqref{LX} and \eqref{GB3} implies
$$  \phi^{-1}(Y_t (x)) =   \rr_t (x)\ \ \text{a.e.}\ (t,x)\in [0,\infty)\times\R^d,$$so that
$$\lim_{k\to\infty} \rr_{n_k,t}(x)=   \rr_t(x), \ \ \text{a.e.}\ (t,x)\in [0,\infty)\times\R^d.$$
Combining this with   $u_{n_k,t}= \rr_{n_k,t}/\rr_{V_{n_k}}$ and $u_t=\rr_t/\rr_V$ due to \eqref{GB4},
we obtain \eqref{GB40}.

 {\bf (4)}    For any $f\in C_0^\infty(\R^d)$, we find $m\ge 1$ such that $ \mathrm{supp} f\subset  B_m$. Then
   for $n_k\ge m+1$ such that $V|B_{m}= V_{n_k}|B_m$ due to \eqref{VNN},
 \eqref{rvr} and \eqref{AN} imply
  \begin{align*}
&\int_{\R^d}\rho_{n_k,t}(x)f(x)\d x-\int_{\R^d}\rho_{n_k,0}(x)f(x)\d x\\
& =-\int_0^t\d s \int_{B_{n_k}}\Big(A_{n_k}(\rho_{n_k,s})
\< \nabla\rho_{n_k,s} + \rho_{n_k,s}\nabla V, \nabla f\>\Big)(x) \d x \\
&=- \int_0^t\d s \int_{B_{n_k}}\Big(\ff 1 {n_k} \rr_{V_{n_k} } \nn u_{n_k,s},\ \nn f\big\> \Big)(x)\d x\\
&\quad  -\int_0^t \d s\int_{\R^d}  \big\<\rr_{V_{n_k}} {\nn \psi_{n_k}}(u_{n_k,s})+  {\nn \phi_{n_k}(\rr_{n_k,s}} )+ \rr_{n_k,s} {\phi_{n_k}'}(\rr_{n_k,s})    \nn V, \ \nn f \big\>(x)   \d x.\end{align*}
By integration by parts formula, we have
\beg{align*} &- \int_{\R^d}  \big\<\rr_{V_{n_k}}\nn \psi_{n_k}(u_{n_k,s})+ \nn \phi_{n_k}(\rr_{n_k,s} )+ \rr_{n_k,s}\phi_{n_k}'(\rr_{n_k,s})    \nn V, \ \nn f \big\>(x)   \d x\\
&= \int_{\R^d} \Big(\big(\phi_{n_k}(\rr_{n_k,s}) +\rr_{V_{n_k}}\psi_{n_k}(u_{n_k,s})   \big)  \DD f\\
 &\qquad\quad-   \big(\rr_s\phi_{n_k}'(\rr_{n_k,s}) +\psi_{n_k}(u_{n_k,s})  \rr_{V_{n_k}}\big) \<\nn V,\nn f\>  \big)\Big)(x) \d x.\end{align*}
So,
\beq\label{LSS} \begin{split}
&\int_{\R^d}\rho_{n_k,t}(x)f(x)\d x-\int_{\R^d}\rho_{n_k,0}(x)f(x)\d x\\
&=- \int_0^t\d s \int_{B_{n_k}}\Big(\ff 1 {n_k} \rr_{V_{n_k} } \nn u_{n_k,s},\ \nn f\big\> \Big)(x)\d x\\
&\quad + \int_0^t \d s\int_{\R^d} \Big(\big(\phi_{n_k}(\rr_{n_k,s}) +\rr_{V_{n_k}}\psi_{n_k}(u_{n_k,s})   \big)  \DD f \Big)(x)\d x\\
&\quad -   \int_0^t \d s\int_{\R^d} \Big(\big(\rr_s\phi_{n_k}'(\rr_{n_k,s}) +\psi_{n_k}(u_{n_k,s})  \rr_{V_{n_k}}\big) \<\nn V,\nn f\>  \big)\Big)(x) \d x.
 \end{split}\end{equation}
 Noting that \eqref{AN}     implies
 \beg{align*} &\int_0^t\d s \int_{B_{n_k}} \ff 1 {n_k} (\rr_{V_{n_k} } |\nn u_{n_k,s}  |)(x)\d x= \int_0^t\d s \int_{B_{n_k}} \ff 1 {n_k} (\rr_{V_{n_k} } |\nn u_{n_k,s}  |)(x)\d x\\
 &\le \bigg(\int_0^t\d s \int_{B_{n_k}} \ff 1 {n_k}     u_{n_k,s}  \d \hat \mu_{n_k} \bigg)^{\ff 1 2}\bigg(\int_0^t\d s \int_{B_{n_k}}
    \Big( A_{n_k}(\rr_{n_k,s} ) \ff{  |\nn u_{n_k,s}  |^2}{u_{n_k,s}} \Big) \d \hat\mu_{n_k}  \bigg)^{\ff 1 2}, \end{align*}
     which goes to $0$ as $k\to\infty$ according to \eqref{LX} and   \eqref{goo}.
By combining this with \eqref{VRL}, \eqref{PPSN},  \eqref{LX}, \eqref{GB40}  and  the dominated convergence theorem,     we may  let    $k\to\infty$ in \eqref{LSS} to derive
\begin{align*}
&\int_ {\R^d}\rho_{t}(x)f(x)\d x - \int_{B_m}\rho_{0}(x)f(x)\d x\\
&=    \int_0^t \d s\int_{\R^d} \Big(\big(\phi(\rr_s) +\rr_V\psi(u_s)   \big)  \DD f -   \big(\rr_s\phi'(\rr_s) +\psi(u_s)  \rr_V\big) \<\nn V,\nn f\>  \big)\Big)(x) \d x.\end{align*}

\end{proof}

\section{Proof of Theorem \ref{TT2} }
 {Recall that $\mu_n$   is defined in \eqref{mun}. To apply  the functional inequalities \eqref{PC10} and \eqref{A4} for $\mu_n$, we  normalize $u_{n,t}$ under $\mu_n$ as follows:
  $$f_{n,t}:= \ff{u_{n,t}}{\theta_{n,t}},\ \ \  \theta_{n,t}:= \mu_n\big(u_{n,t}\big).$$
  By   \eqref{VRL} and $\hat\mu_n(u_{n,t})=1$, we have
\beq\label{H1}  s_n^{-1} \mu_n\le \hat\mu_n\le s_n \mu_n, \ \ \ s_n^{-1} \le \theta_{n,t}\le s_n,\ \ \ \mu_n(f_{n,t})=\hat\mu_n(u_{n,t})=1,\ \ t\ge 0,  n\ge 1.\end{equation}
Combining this with \eqref{GG} and  \eqref{VRL}, we find $1<\gg_n\downarrow 1$ as $n\uparrow\infty$ such that
\beq\label{H2} \phi'(f_{n,t}\rr_V) +\psi'(f_{n,t})\le \gg_n \big(\phi'(\rr_{n,t} ) +\psi'(u_{n,t})\big),\ \ n\ge 1,\ t\ge 0.\end{equation}
 {Indeed,   \eqref{GG} implies
\begin{align*}
\phi'(f_{n,t}\rr_V)=\phi'\Big(\ff{\rho_{n,t}}{\theta_{n,t}}\frac{\rr_V}{\rho_{V_n}}\Big)\leq \sup_{s\in [s_n^{-2}, s_n^2]}\phi'(s\rho_{n,t})\leq \gamma(s_n^2)\phi'(\rho_{n,t}),
\end{align*}
\begin{align*}
\psi'(f_{n,t})=\psi'\Big(\ff{u_{n,t}}{\theta_{n,t}}\Big)\leq \sup_{s\in [s_n^{-1}, s_n]}\psi'(su_{n,t})\leq \gamma(s_n)\psi'(u_{n,t}).
\end{align*}
So, \eqref{H2} holds for $\gamma_n=\max\big\{\gamma(s_n),\gamma(s_n^2)\big\}$ which goes to $1$ as $n\to\infty$.

The first inequality in \eqref{H1} will be frequently used in the proof of Theorem \ref{TT2}, since the invariant probability of \eqref{rvr} is $\hat{\mu}_n$ but the given    functional inequalities \eqref{PC10} and \eqref{A4} associates to $\mu_n$.
}
\beg{proof}[Proof of Theorem \ref{TT2} (1)]  We consider $p\in (1,\infty)$ and $p=1$ respectively.

 (a) Let $p\in (1,\infty)$.  To apply \eqref{PC10} to $u_t$ and $f_{n,t}$, we first make them uniformly positive as required:
   $$u_{t}^{(\vv)}:=\ff{u_t+\vv}{1+\vv},\ \ \ u_{n,t}^{(\vv)}:= \ff{u_{n,t}+\vv}{1+\vv},\ \ \ f_{n,t}^{(\vv)}:= \ff{u_{n,t}+\vv}{\theta_{n,t}+\vv},\ \ \ \vv\in (0,1),\ n\in\mathbb N,\ t\ge 0.$$
   Similarly to \eqref{H2}, we find constants $1<\gg_{n,\vv}\to 1$ as $(n,\vv)  \to (\infty,0)$, such that
   \beq\label{H2'}  \phi'(f_{ n,t}^{(\vv)}\rr_V) +\psi'(f_{n,t}^{(\vv)})\le \gg_{n,\vv} \big(\phi'(\rr_{n,t} ) +\psi'(u_{n,t})\big).\end{equation}
 By \eqref{H1}, \eqref{H2'} and
  \eqref{ERG},   we obtain
 \beq\label{EE*} \beg{split}&  \partial_t {\rm Var}_{p,\hat\mu_n} \big(u_{n,t}^{(\vv)}\big)= \ff 1 {p-1} \pp_t \hat\mu_n\Big( \big(u_{n,t}^{(\vv)}\big)^p \Big)\\
 & = -  p  \hat\mu_n\Big( A_n(\rho_{n,t})\big(u_{n,t}^{(\vv)}\big)^{p-2} \big|\nabla u_{n,t}^{(\vv)} \big|^2 \Big)   \\
 &\le -  \ff {p}{\gg_{n,\vv} s_n^{p+1}}   \mu_n \Big(\big[\phi'(f_{n,t}^{(\vv)}\rr_{V})+\psi'(f_{n,t}^{(\vv)})\big] \big(f_{n,t}^{(\vv)}\big)^{p-2} \big|\nn f_{n,t}^{(\vv)}\big|^2\Big),\\
 &\qquad\qquad  \ t\ge 0,\ n\ge 1,\ \vv\in (0,1).\end{split}\end{equation}
 Combining this with \eqref{PC10} and \eqref{H1}, we derive
\beq\label{H3} \beg{split}& \partial_t {\rm Var}_{p,\hat\mu_n} \big(u_{n,t}^{(\vv)}\big) \le -  \ff {\ll}{\gg_{n,\vv} s_n^{p+1}}  \mu_n\Big(\big(f_{n,t}^{(\vv)}\big)^p  -1 \Big) \\
&\le  -  \ff {\ll}{\gg_{n,\vv}s_n^{2p+2} } \hat \mu_n\Big(\big(u_{n,t}^{(\vv)}\big)^p -1\Big) + \ff {\ll}{\gg_{n,\vv}s_n^{2p+2} } \big|s_n^p-1\big|   \\
&=  - \ff {\ll }{\gg_{n,\vv}  s_n^{2p+2}   } {\rm Var}_{p,\hat\mu_n} \big(u_{n,t}^{(\vv)}\big)  + \ff {\ll}{\gg_{n,\vv} s_n^{2p+2} } \big|s_n^p  -1\big|.\end{split}\end{equation}
 Then
\beq\label{H44}  {\rm Var}_{p,\hat\mu_n} \big(u_{n,t}^{(\vv)}\big)  \le \exp\bigg[-\ff {\ll t}{\gg_{n,\vv} s_n^{2p+2}   }    \bigg] {\rm Var}_{p,\hat\mu_n}  \big(u_{n,0}^{(\vv)}\big)  +
 \big|s_n^p- 1\big|,\ \  n\ge 1,\ t\ge 0. \end{equation}

On the other hand,    by  \eqref{VRL}, \eqref{W*},  \eqref{H1} and $s_{n_k}\to 1$ as  $k\to\infty$, for any  $f\in \B_b(\R^d)$ with $\mu(|f|^{\ff p{p-1}})\le 1,$ we have
 \beg{align*}&\Big|\mu \big( u_t^{(\vv)} f\big) \Big|^p  =\lim_{k\to\infty} \Big|\hat \mu_{n_k}  \big(u_{n_k, t}^{(\vv)} f\big) \Big|^p  \\
 &\le  \liminf_{k\to\infty}    \hat\mu_{n_k}\Big(\big(u_{n_k, t}^{(\vv)}\big)^p\Big)    =  \liminf_{k\to\infty}    {\rm Var}_{p,\hat\mu_{n_k}} \big(u_{n_k,t}^{(\vv)}\big) +1.   \end{align*}
 This  implies
$$  {\rm Var}_{p, \mu} \big(u_t^{(\vv)} \big) \le   \liminf_{k\to\infty}    {\rm Var}_{p,\hat\mu_{n_k}} \big(u_{n_k,t}^{(\vv)}\big), $$ so that
by \eqref{H44} and that  $s_{n_k},\gg_{n_k,\vv}\to 1$ as $k\to\infty$ and $\vv\to 0$, we may first letting $k\to\infty $ then $\vv\to 0$ to derive
\beq\label{X*3}\liminf_{\vv\to 0}  {\rm Var}_{p, \mu} \big(u_t^{(\vv)} \big)\le \e^{- \ll t}  \liminf_{\vv\to 0}  \liminf_{k\to\infty}   {\rm Var}_{p,\hat\mu_{n_k}}  \big(u_{n_k, 0}^{(\vv)}\big),\ \ \ \vv\in (0,1),\ t\ge 0.\end{equation}

Moreover, by  \eqref{VRL},  \eqref{intbo} and the dominated convergence theorem,
we have
$$\liminf_{\vv\to 0} \lim_{k\to\infty}   {\rm Var}_{p,\hat\mu_{n_k}} \big(u_{n, 0}^{(\vv)}\big)  =  \lim_{\vv\to 0}   {\rm Var}_{p, \mu} \big(u_0^{(\vv)}\big)
=  {\rm Var}_{p, \mu}  (u_0).$$
Combining this with \eqref{X*3}    we derive   \eqref{A2}.

(b) Let $p=1$ so that ${\rm Var}_{1,\mu}(f)= \mu(f\log  f) $ for $f\ge 0$ with $\mu(f)=1$. Then \eqref{PC10}  becomes  the nonlinear log-Sobolev inequality
\beq\label{LS}  \beg{split}  \mu_n(f\log f) \le &\ff 1\ll \mu_n\big(\big[\phi'(f\rr_{V} )+\psi'(f)\big]  f^{-1} |\nn f|^2\big),\\
&\quad  \ \  f\in C_b^1(B_n),\ \inf_{B_n} f>0,\ \mu_n( f)=1,\end{split} \end{equation}
and  \eqref{A2} reduces to the exponential convergence in entropy
\beq\label{ENT} \mu(u_t\log u_t)\le \e^{-\ll t} \mu(u_0\log u_0),\ \ \ t\ge 0.\end{equation}
   By \eqref{EN*}, \eqref{H1} and \eqref{H2'}, we obtain
   \beg{align*} &-\pp_t E_{n,\vv}(t) \ge \hat\mu_n\bigg(\Big(\phi'(\rr_{n,t} ) + \psi'(u_{n,t})\Big) \ff{|\nn u_{n,t}|^2}{u_{n,t}+\vv}\bigg)\\
   &\ge  \ff 1 {s_n\gg_{n,\vv}(\theta_{n,t}+\vv)} \mu_n\bigg(\Big(\phi'(f_{n,t}^{(\vv)} \rr_V) + \psi'(f_{n,t}^{(\vv)})\Big) \ff{|\nn f_{n,t}^{(\vv)}|^2}{f_{n,t}^{(\vv)}}\bigg).\end{align*}
   This together with \eqref{LS} implies
  \beq\label{H5} - \pp_t E_{n,\vv}(t) \ge  \ff \ll {s_n\gg_{n,\vv}(\theta_{n,t}+\vv)}  \mu_n\big(f_{n,t}^{(\vv)}\log f_{n,t}^{(\vv)}\big),\ \ \ t\ge 0.\end{equation}
    Note that
      $$r\log r - r(1+\log c)+c\ge 0,\ \ \ r\ge 0, \ c>0,$$ and for any probability space $(E,\F,\nu),$
  $$\nu(f\log f) -\nu(f)\log \nu(f)= \inf_{c>0} \nu\big( f\log f- f(1+\log c) +c\big),\ \ \ 0\le f\in L^1(\nu).$$
Combining these with \eqref{H1} and $\mu_n(f_{n,t})=1$, we obtain
  {\beg{align*} &  \mu_n\big(f_{n,t}^{(\vv)}\log f_{n,t}^{(\vv)}\big) = \inf_{c>0} \mu_n\big( f_{n,t}^{(\vv)}\log f_{n,t}^{(\vv)}- f_{n,t}^{(\vv)}(1+\log c) +c\big)\\
 &\geq \ff 1{s_n(\theta_{n,t}+\vv)} \inf_{c>0}   \hat{\mu}_n\bigg( (u_{n,t}+\vv) \log (u_{n,t}+\vv) -  \big(1+\log [ c (\theta_{n,t}+\vv)] \big) +c (\theta_{n,t}+\vv)\bigg)\\
 &=  \ff 1{s_n(\theta_{n,t}+\vv)} \Big(E_{n,\vv}(t) +\vv \hat{\mu}_n(\log (u_{n,t}+\vv)) \Big).\end{align*}}
 This together with \eqref{H5} and \eqref{H1} yields
 $$\pp_t E_{n,\vv}(t)\le - \ff \ll {s_n^2 \gg_{n,\vv}(s_n+\vv)^2}   \Big(E_{n,\vv}(t) +\vv  {\hat{\mu}_n\big(\log (u_{n,t}+\vv)\big)} \Big),\ \ \ t\ge 0,$$
 so that
$$ E_{n_k,\vv}(t)\le \exp\bigg[-\ff {\ll t} {s_{n_k}^2 \gg_{n_k,\vv}(s_{n_k}+\vv)^2} \bigg] E_{n_k,\vv}(0) + \vv  {\big|\hat{\mu}_n(\log (u_{n,t}+\vv))\big|}.$$
By \eqref{intbo}, {\eqref{LX}} and that $s_{n_k}, \gg_{n_k,\vv}\to 1$ as $(k,\vv)\to (\infty,0)$,  by first letting $\vv\to 0$ then $k\to \infty$
we derive
$$  \liminf_{k\to\infty} \hat\mu_{n_k} (u_{n_k,t}\log u_{n_k,t}) \le \e^{- \ll t }   \mu(u_0\log u_0).$$
To prove \eqref{ENT}, it remains to verify
\beq\label{LMT} \mu(u_t\log u_t)\le   \liminf_{k\to\infty} \hat\mu_{n_k} (u_{n_k,t}\log u_{n_k,t}).\end{equation}
 To this end, we use the variational formula, see e.g. \cite[Proof of Theorem 1.4.2(2)]{W13},
 $$\mu(u_t\log u_t)= \sup_{f\in \B_b(\R^d)} \Big[\mu\big(u_t  f\big)-\log \mu\big(\e^f\big)\Big].$$
Combining this with \eqref{W**} and \eqref{H1},  \eqref{VRL} and $\|u_{n_k,t}\|_\infty\le M_0$ due to {\eqref{LX}}, we obtain
\beg{align*} &\mu(u_t\log u_t)= \sup_{f\in \B_b(\R^d)} \lim_{k\to\infty}  \Big[\hat\mu_{n_k}\big( u_{n_k,t} f\big)-\log \hat\mu_{n_k}\big(\e^f\big) \Big]\\
&\le \liminf_{k\to\infty} \sup_{f\in \B_b(\R^d)}  \Big[\hat\mu_{n_k}\big(  u_{n_k,t}   f\big)-\log \hat\mu_{n_k}\big(\e^f\big) \Big]\\
&= \liminf_{k\to\infty} \hat\mu_{n_k}\big(u_{n_k,t}\log u_{n_k,t}\big).   \end{align*}
So, \eqref{LMT} holds and the proof is finished.
\end{proof}

 \beg{proof}[Proof of Theorem \ref{TT2} (2)]
  Note that $\|u_{n,t}\|_\infty\le \|u_{n,0}\|_\infty$ due to Theorem \ref{pri}. Similarly to \eqref{H3} but use \eqref{A4} in place of \eqref{PC10}, we obtain
 \beg{align*}& \partial_t \hat\mu_n\big(|u_{n,t}-1|^2 \big)\le -  \ff {2}{\aa(r)\gg_n s_n^3} \mu_n\big(|f_{n,t}-1|^2\big)  + r\|u_{n,0}\|_\infty^2  \\
 &\le - \ff {2 (1-\vv)}{\aa(r)s_n^7\gg_n }   \hat\mu_n\big(|u_{n,t}-1|^2\big) + \ff {2(1+\vv^{-1})}{\aa(r)s_n^5\gg_n}  \big|s_n- 1\big|^2
 +  r\|u_{n,0}\|_\infty^2,\ \ \ \vv\in (0,1).\end{align*}
Then
\beg{align*}&\hat\mu_n\big(|u_{n,t}-1|^2 \big)\le \exp\bigg[-\ff {2 (1-\vv)t}{\aa(r)s_n^7\gg_n }  \bigg]  \hat\mu_n\big(|u_{n,0}-1|^2 \big)\\
&\qquad  +\ff {s_n^2(1+\vv^{-1})}{1-\vv} \big|s_n- 1\big|^2+ r \|u_{n,0}\|_\infty^2,\ \ \ \vv\in (0,1),\ r>0.\end{align*}
 Combining this with  \eqref{intbo},  by using $n_k$ in place of $n$ and letting $k\to\infty$ and $\vv\to 0$, we obtain
\beq\label{H4}  \mu\big(|u_t-1|^2 \big)\le \e^{-2t/\aa(r)} \mu(|u_0-1|)^2 + r    \|u_{0}\|_\infty^2, \ \ r>0,\ t\ge 0.  \end{equation}
This implies   \eqref{A5} holds. \end{proof}

\beg{proof}[Proof of Theorem \ref{TT2} (3)]     Let   $\eqref{A4}$ hold with $\aa(r)= \dd r^{-\theta}$ for $r\in (0,1]$.   By taking $k>0$ such that
  $$ \e^{-k^{\theta}/\dd} = 2^{-(2+\theta^{-1})},$$ by $\|u_t\|_\infty\le \|u_0\|_\infty$ due to Theorem \ref{pri}, and applying   \eqref{H4}   for the time interval $[t/2,t]$
  $$r= r(t):= k t^{-\ff 1\theta},$$    we obtain
\beg{align*}  \mu (|u_t-1|^2) &\le  \e^{-t r(t)^\theta/\dd} \mu (|u_{t/2}-1|^2)   + r(t)   \|u_{t/2}\|_\infty^2\\
 &\le \mu (|u_{t/2}-1|^2)2^{-(2+\theta^{-1})}  + k t^{-\ff 1\theta}  \|u_{0}\|_\infty^2,\ \ \ t\ge 0. \end{align*}
  Applying this inequality for $t/2$ in place of $t$ and repeating,     we conclude that for any $n\in \mathbb N$,
 \beg{align*}  \mu (|u_t-1|^2) &\le  \mu (|u_{t2^{-n}}-1|^2) 2^{-n(2+\theta^{-1})} + k   \|u_{0}\|_\infty^2 \sum_{i=0}^{n-1} (2^{-i}t)^{-\ff 1\theta}  2^{-i(2+\theta^{-1})}.
  \end{align*}
Since  $\|u_{t2^{-n}} \|_\infty\le \|u_0\|_\infty<\infty$, by  letting $n\to\infty$  we derive   \eqref{A6}. \end{proof}

 \section{Application to  the non-degenerate case }

 In this part, we consider the following non-degenerate case.

\beg{enumerate}  \item[$(A_1)$] {\bf (A)} and {\bf (B)} hold, and there exists a constant $c_0>0$ such that
$$\psi'(s)+\phi'(t)\ge c_0,\ \ \  s,t\ge 0.$$\end{enumerate}
Typical examples include
$$\psi(s)= a_1 s+ a_2 s^p,\  \ \ \phi(s)= b_1 s+ b_2 s^q$$ for constants $a_1,a_2, b_1, b_2\ge 0$ and $p,q>1$ with $c_0:= a_1+b_1>0.$

Under $(A_1)$,   we have
$$c_0  |\nn f|^2 \le  \mu    \big[\phi'(f\rr_{V})+\psi'(f)\big]   \big|\nn  f \big|^2,\ \ 0\le f\in C_b^1(B_n),$$
so that the following is a direct consequence of Theorem \ref{TT2} where the functional inequalities reduce to corresponding ones in the linear setting.

 \beg{cor}\label{C2.1}  Assume   $(A_1)$. Then $\eqref{E1}$ has  a distributional solution with initial value $\rr_0$ such that the following assertions hold.
 \beg{enumerate} \item[$(1)$] Let $p\in [1,\infty)$. If there exists a constant $\ll >0$ such that for large $n\ge 1$
 \beq\label{EI} {\rm Var}_{p,\mu_n}    (f) \le \ff p \ll \mu_n\big(f^{p-2}   |\nn f|^2\big),\ \ \ f\in C_b^1(B_n),\ \inf_{B_n} f> 0,\  {\mu_n(f)=1},\end{equation}
 then
\beq\label{EX13}  {\rm Var}_{p,\mu} \big( u_t \big)\le \e^{-\ll c_0 t }\,  {\rm Var}_{p,\mu} \big(u_0\big),\ \ \ t\ge 0.\end{equation}
\item[$(2)$] If there exists a decreasing $\aa: (0,\infty)\to (0,\infty)$  such that the following weak Poincar\'e  inequality hods for large $n\ge 1$:
\beq\label{AB4'}\beg{split}  & \mu_n  ( |f-1|^2)\le  \aa(r) \mu_n  \big(  |\nn  f  |^2\big)+ r  \|f\|_\infty^2,\\
&\ \ \qquad  \ 0\le f\in C_b^1(B_n),\  {\mu_n(f)=1}, \ r>0,\end{split}  \end{equation}
then
$$    \mu(|u_t-1|^2)  \le  \inf_{r>0} \Big\{ {\e^{-2c_0t/\aa(r)}} + r  \Big\}  \|u_0-1\|_\infty^2,\ \ \ t> 0. $$
\item[$(3)$] If $\eqref{AB4'}$  holds  for large $n\ge 1$ and $\aa(r)= \dd (1+r^{-\theta})$   for  some constants $\dd,\theta>0$, then there exists a constant $k>0$ such that
$$   \mu(|u_t-1|^2) \le k  t^{-\ff 1\theta} \|u_0-1\|_\infty^2,\ \ \ t>0.$$

\end{enumerate}
\end{cor}

 Below we   illustrate Corollary \ref{C2.1} using explicit conditions  on $V$.

 \beg{cor} \label{C2.2}   Assume   $(A_1)$ and let $\kk\in (0,\infty)$.  Then $\eqref{E1}$ has  a distributional solution with initial value $\rr_0$ such that the following assertions hold.
  \beg{enumerate}
 \item[$(1)$]  If $\Hess_V\ge \kk$ for some constant $\kk>0$, then
\beq\label{EX}  \mu  (u_t\log u_t)\le \e^{- 2c_0\kk  t}  \mu (u_0\log u_0),\ \ t\ge 0,\end{equation}
  \beq\label{EX2}  \mu(|u_t-1|^2)\le \e^{- 2c_0 \kk t}  \mu(|u_0-1|^2),\ \ \ t\ge 0.\end{equation}
  \item[$(2)$]   More generally, if $V= W_1+W_2+W_3$ with $\|W_1\|_\infty+\|\nn W_2\|_\infty<\infty$ and $\Hess_{W_3}\ge  \kk>0$  holds outside a compact set, then   \eqref{EX} and \eqref{EX2} hold for some constant $\kk'>0$ in place of $\kk.$
  \end{enumerate} \end{cor}
   \beg{proof}
(1) By Bakry-Emery's criterion \cite{BE} which also holds on compact domain, $\Hess_V\ge \kk$ implies the  log-Sobolev inequality
 {\beq\label{LS'} \mu_n(f^2\log f^2)\le \ff{2}{\kk} \mu_n(|\nn f|^2),\ \ \ f\in C_b^1(\R^d),\ \mu_n(f^2)=1,\end{equation}}
which is equivalent to \eqref{EI} for $p=1$ and $\ll=2\kk$.
Indeed,  since  the Neumann semigroup generated by $L_V:=\DD-\nn V\cdot\nabla $ on $B_n$ is ergodic with invariant probability measure $\mu_n$,  \eqref{LS'} follows from the following semigroup log-Sobolev inequality in  \cite[Theorem 3.3.2(10)]{Wbook} by letting $t\to\infty$. So, \eqref{EX} follows from   {Corollary \ref{C2.1}(1) for $p=1$}.

According to \cite{DS},   see also \cite[Proposition 5.1.3]{BGL},
 \eqref{LS'}   implies the Poincar\'e inequality
 { \beq\label{PC1} {\rm Var}_{\mu_n}(f)\le \ff 1 \kk \mu_n(|\nn f|^2),\ \ f\in C_b^1(\R^d),\end{equation}}
  which coincides with \eqref{EI} for $p=2$ and $\ll= 2\kk$,     so that    Corollary \ref{C2.1}(1) implies \eqref{EX2}  {  for $p=2$}.

 (2) According to \cite{A, DS2}, the log-Sobolev inequality is stable under perturbations of $V$ which are bounded or Lipschitz continuous,
 we may simply  assume $W_1=W_2=0$ so that   $\Hess_V\ge \kk>0$ holds outside a compact set.
  In this case,   by Theorem \cite[Theorem 1.1]{W09},    the log-Sobolev inequality
 \eqref{LS'} holds for any $n\ge 1$ and some constant $\kk'>0$ in place of $\kk$, which implies,  as explained above,    the estimates   \eqref{EX} and \eqref{EX2}  for $\kk'$ in place of $\kk$.
\end{proof}

Next,  we consider $V$ such that    the log-Sobolev inequality may be not available  but  the Poincar\'e inequality holds.

 \beg{cor} \label{C2.3}  Assume   $(A_1)$. Then $\eqref{E1}$ has  a distributional solution with initial value $\rr_0$ such that the following assertions hold.
   \beg{enumerate}
  \item[$(1)$] Let $V(x)=W(x)+  V_0(x)$ for some   bounded $W $ and $V_0\in C^1(\R^d)$.    If  {\beq\label{LY}  \liminf_{|x|\to\infty} \ff 1 {|x|}\<\nn V_0(x),x\>  >0,\end{equation}}
 then for some   constant $\ll >0$,
   $$ \mu(|u_t-1|^2)\le \e^{- \ll   t}  \mu (|u_0-1|^2),\ \ \ t\ge 0.$$
  \item[$(2)$] Let $V(x)=W(x)+  \ll |x|^q$ for some constants $\ll>0, q\in [1,2)$ and $W\in C_b(\R^d)$. Then there exists a constant $c>0$ such that
  $${\rm Var}_{p,\mu}  (u_t)\le \exp\Big[- c (p-1)^{\ff{2-q}q}  t\Big] {\rm Var}_{p, \mu}  (u_0),\ \ t\ge 0,\ p\in (1,2].$$
  \end{enumerate}
   \end{cor}

  \beg{proof}   (1)   By \cite[Corollary 1.4]{W99}, if \eqref{LY} holds,   then there exists a constant $\ll >0$ such that for any $n\ge 1$ the Poincar\'e inequality
\eqref{EI} with $p=2$ holds for $\bar{\mu}(\d x):= Z(V_0)^{-1} \e^{-V_0(x)}\d x$ in place of $\mu$. Since $W$ is bounded, by the standard perturbation result of Poincar\'e inequality,
\eqref{EI} holds for $\mu$ and  a different constant $\ll>0$. Then the desired estimate follows from Corollary \ref{C2.1} $(1)$ for  {$p=2$}.

(2) Let $p\in (1,2]$. Using  {$\ff{|f|^{\ff 2p}}{\mu(|f|^{\ff 2p} )}$} in place of $f$, we see that \eqref{EI}  is equivalent to  the $p$-variance inequality
  {  \beq\label{PV}   \mu_n(f^2)-\mu_n(|f|^{\ff 2 p})^{p} \le \ff{4 (p-1)}{\lambda p}\mu_n(|\nn f|^2),\ \   f\in C_b^1(\R^d).\end{equation}}
By \cite[Corollary 6.2.3]{Wbook}, for the given $V$,  {\eqref{PV} holds for any $p\in (1,2]$ and $\lambda=k_0(p-1)^{\frac{2-q}{q}}$ with some constant $k_0>0$,}
  so that    the desired estimate follows from Corollary \ref{C2.1} $(1)$. \end{proof}

 Moreover,  we consider examples  where the Poincar\'e inequality fails, and    $ \mu(|u_t-1|^2)$ decays in there typical rates: sub-exponential, algebraic and logarithmic.

    \beg{cor} \label{C2.4}  Assume  $(A_1)$. Then $\eqref{E1}$ has  a distributional solution with initial value $\rr_0$ such that  the following assertions hold.
 \beg{enumerate}
      \item[$(1)$]  Let $V(x)=W(x)+  \lambda |x|^q$ for some constants $\lambda>0, q\in (0,1)$ and bounded $W$. Then there exists a constant $k_1,k_2>0$ such that
 \beq\label{WE1}  \mu(|u_t-1|^2)\le      k_1 \big\|u_0\big\|_\infty^2 \exp\Big[- k_2 t^{\ff q {4-3q}}\Big],\ \ \ t\ge 0.\end{equation}
  \item[$(2)$]    Let $V(x)=W(x)+ (d+q) \log (1+|x|)$ for some constant $q>0$ and  bounded $W$. Then there exists a constant $k >0$ such that
 \beq\label{WE2} \mu(|u_t-1|^2) \le      k  \big\|u_0\big\|_\infty^2 t^{- \tau},\ \ \ t> 0, \end{equation}
 where
 $$\tau:=\max\Big\{\ff q{d+2+q},\ \ff{q^2-4-2d-2q}{4q+4+2d}\Big\}.$$
    \item[$(3)$]        Let $V(x)=W(x)+  d  \log (1+|x|) +q\log\log (\e+ |x|)$ for some constant $q>1$ and  bounded $W$.  Then there exists a constant $k >0$ such that
 \beq\label{WE3}\mu(|u_t-1|^2) \le       k \big\|u_0\big\|_\infty^2  [\log(1+ t)]^{1-q},\ \ \ t> 0. \end{equation}
       \end{enumerate} \end{cor}
\beg{proof} By Corollary \ref{C2.1} $(2)$-$(3)$, it suffices to verify \eqref{AB4'} for suitable $\aa(r)$.
 Since this inequality   trivially holds for $\aa(r)=0$ and $r\ge 1$,
we only need to consider  $r\in (0,1].$

 (1)  By the argument in the proof of \cite[Example 1.4(c)]{RW01}, there exists a constant $c>0$ such that  for any $n\ge 1$, the inequality \eqref{AB4'}
 holds for
 $$\aa(r):= c\big[\log (1+r^{-1})\big]^{\ff{4(1-q)}q},\ \ \ r\in (0,1].$$
 So,  by Corollary \ref{C2.1}(2), we derive \eqref{WE1} for some constants $k_1,k_2>0.$

 (2)    For $V=W+(d+q)\log (1+|x|)$ with bounded $W$, by the proof of \cite[Example 1.4(a)]{RW01},   we find a constant $c>0$ such that \eqref{AB4'} holds for
  any $n\ge 1$ and $\aa(r)= c r^{-\ff 1\tau}$.
  Therefore,   by Corollary \ref{C2.1}(3), we derive \eqref{WE2} for some constant  $k >0.$

 (3) By  the proof of \cite[Example 1.4(b)]{RW01},   there exists a constant $c>0$ such that \eqref{AB4'}
 holds for and $n\ge 1$ and
 $$\aa(r):= \exp\big[c r^{-\ff 1 {q-1}}\big],  \ \ \ r\in (0,1].$$
 So,  by Corollary \ref{C2.1}(2), we derive \eqref{WE3} for some constant  $k >0.$
  \end{proof}

  \section{Application to  the  degenerate case }

In this part, we consider the degenerate case under the following assumption.
\beg{enumerate}  \item[$(A_2)$]  {\bf (A)} and {\bf (B)} hold,     and $K_\psi(r)+K_\phi(r)>0$ for $r>0$, where
$$K_\psi(r):=\inf_{s\ge r}\psi'(s),\ \ \ K_\phi(r):=\inf_{s\ge r}\phi'(s).$$
\end{enumerate}

 \subsection{Estimate on the convergence rate using  $K_\psi$.}

\beg{cor}\label{C4.1} Assume   $(A_2)$ with
   $K_\psi(r) >0$ for $r>0$. Then $\eqref{E1}$ has  a distributional solution with initial value $\rr_0$ such that the following assertions hold.
   \begin{enumerate}
   \item[$(1)$]Let  $\eqref{EI}$  {with $p=2$} hold  for some constant $\ll>0$ and  large $n\ge 1$. Then
   \beq\label{B1} \mu (|u_t-1|^2)\le \inf_{r>0}\Big\{\e^{- \ll K_\psi(r/3) t } + r\Big\} \|u_0\|_\infty^2,\ \ t>0.\end{equation}
  If $K_\psi(r)\ge \dd r^{\theta}$ for $r\in (0,1]$ and some constants $\dd,\theta >0$, then there exists a constant $k>0$ such that
\beq\label{B1'}  \mu (|u_t-1|^2)\le   k t^{-\ff 1 \theta} \|u_0\|_\infty^2,\ \ t>0.\end{equation}
\item[$(2)$] { If the weak Poincar\'e inequality \eqref{AB4'} holds for large $n\ge 1$, then there exists a constant $k>0$ such that
 \beq\label{B3}   \mu(|u_t-1|^2)  \le k \inf_{r>0} \Big\{\e^{-\ff{K_\psi(r/4)}{ \aa(r/4)}t } + r\Big\} \|u_0\|_\infty^2,\ \ \ t> 0.
 \end{equation}
If $\ff{K_\psi(r)}{\aa(r)}\ge \dd r^{\theta}$ for $r\in (0,1]$ and some constants $\dd,\theta >0$, then there exists a constant $k>0$ such that
 \beq\label{B4} \mu(|u_t-1|^2) \le k  t^{-\ff 1\theta} \|u_0\|_\infty^2,\ \ \ t>0.
 \end{equation}}
\end{enumerate}
 \end{cor}

\beg{proof}  Let ${\rm Var}_{\mu_n} (g):= \mu_n(g^2)- \mu_n(g)^2, \  g\in L^2(\mu_n).$
For any $f\in C_b^1(\R^d)$ with $f\ge 0,  {\mu_n(f)=1},$ we have
$$ \mu_n(f\lor r)^2 -1\le (1+r)^2-1\le 2r +r^2\le 3 r,\ \ r>0,$$ so that
\beq\label{HH} \beg{split}  & \mu_n (|f-1|^2) = \mu_n(f^2)-1\le  \mu_n ((f\lor r)^2)-1\\
&\le {\rm Var}_{\mu_n} (f\lor r )  + 3r,\ \ r>0.\end{split}\end{equation}
Noting that    \eqref{EI}  {with $p=2$} is equivalent to
$$  \mu_n(g^2)-\mu_n(g)^2 \le \ff 2 \ll  \mu_n\big( |\nn g|^2\big),\ \  \  g\in C_b^1(B_n),$$
by     \eqref{HH} and this inequality for $g= (f\lor r)|_{B_n}$,  we obtain
\beg{align*}& \mu_n (|f-1|^2)  \le \ff 2 \ll \mu_n\big( |\nn f|^21_{\{f\ge r\}}\big) + {3r} \\
 &\le \ff{2}{\ll K_\psi(r)} \mu_n\big( (\phi'(f\rr_V)+\psi'(f))|\nn f|^2\big) + 3r,\ \ r\in (0,1].\end{align*} Since  {$\|u_0\|_\infty \ge 1$,}
this implies  \eqref{A4}  for $\aa(r):= \ff{2}{\ll K_\psi(r/3)}$.
So,       \eqref{B1} follows from \eqref{A5}.

 If $K_\psi(r)\ge \dd r^\theta$ for $r\in (0,1]$, then  \eqref{A4}  holds for $\aa(r):=  \dd(1+ r^{-\theta})$ for some   constant $\dd>0$, so that
 {\eqref{B1'} follows from   \eqref{A6}.}

  {(2) If \eqref{AB4'} holds, then   \eqref{HH} implies
\beg{align*} & \mu_n (|f-1|^2)  \le\alpha (r)\mu_n\big( |\nn f|^21_{\{f\ge r\}}\big) +3 r
+r\|f\|_\infty^2\\
 &\le \ff{\aa(r)}{K_\psi(r)} \mu_n\big({(\phi'(f\rr_V)+\psi'(f))}|\nn f|^2\big) +  4 r \|f\|_\infty^2,\ \ r\in (0,1].
 \end{align*}
 Then the proof is finished as explained in step (1).}
 \end{proof}

\beg{rem} When $\phi=0$ and $\psi(s)=s^m$ for some $m>1$,
the algebraic    convergence rate       has been derived in \cite[Theorem 4.5]{DGGW}  for the classical solution
  by using   the $L^{2/(m+1)}$-Poincar\'e and $L^{1/m}$-log-Sobolev    inequalities.
The above Corollary \ref{C4.1} describes a more general convergence rate for distributional solutions of  more general models.
 \end{rem}

Below we consider a special case  $\psi(s)=\kk s^m$ for some constants $\kk>0$ and $m>1$.

\beg{cor}\label{C4.2} Assume  $(A_2)$ where
  $\psi(s)=\kk s^m$ for some constants $\kk>0$ and $m>1$.  Then $\eqref{E1}$ has  a distributional solution with initial value $\rr_0$ such that  the following assertions hold.
\beg{enumerate} \item[$(1)$]
If the Poincar\'e inequality \eqref{EI} holds for large $n\ge 1$, then there exists a constant $k>0$ such that
$$  \mu (|u_t-1|^2)\le k t^{-\ff 1 {m-1}} \|u_0\|_\infty^2,\ \ \ t>0.$$
It is the case for $V$ in Corollary \ref{C2.3}.
\item[$(2)$]  For $V$ in Corollary $\ref{C2.4}(1)$, there exists a constant $k>0$ such that
\beq\label{V1} \mu(|u_t-1|^2)\le  k \big\|u_0 \big\|_\infty^2  t^{-\ff 1 {m-1}}\big[\log(1+t)\big]^{\ff{5q-4}{q(m-1)}},\ \ t>0.\end{equation}
For $V$ in Corollary $\ref{C2.4}(2)$, there exists a constant $k>0$ such that
\beq\label{V2} \mu(|u_t-1|^2)\le  k \big\|u_0 \big\|_\infty^2   {t^{-\ff 1 {m-1+\tau^{-1}}}},\ \ t>0.\end{equation}
For $V$ in Corollary $\ref{C2.4}(3)$, there exists a constant $k>0$ such that
\beq\label{V3} \mu(|u_t-1|^2)\le k \big\|u_0 \big\|_\infty^2   {[\log(1+ t)]^{1-q}},\ \ t>0.\end{equation}
 \end{enumerate}
 \end{cor}

 \beg{proof} It is easy to see that $K_\psi(r)=\kk mr^{m-1}$, and  the Poincar\'e inequality \eqref{EI} holds for $V$ in Corollary \ref{C2.3}.
 So, the first two assertion  follows from Corollary \ref{C4.1}. It remains to prove the last assertion.

 For $V$ in Corollary \ref{C2.4}(1), \eqref{AB4'} holds for any $n\ge 1$ and $\aa(r)= c \big[\log(1+r^{-1})\big]^{\ff{4(1-q)}q}$ for some constant $c>0$. So, taking
 $$r=\ell t^{-\ff 1 {m-1}} \big[\log (1+t)\big]^{\ff 1 {m-1}(1+\ff{4(q-1)}q)}$$
 for a large enough constant $\ell >0$, we find a constant $k(\ell )>0$ such that
 $$\inf_{r>0}\Big\{\e^{-\ff{K_\psi(r)}{\aa(r)}t } + r\Big\}= \inf_{r>0}\Big\{\e^{-\ff{\kk mr^{m-1}}{\aa(r)}t } + r\Big\}\le k(\ell) t^{-\ff 1 {m-1}}\big[\log(1+t)\big]^{\ff{5q-4}{q(m-1)}}.$$
 So, \eqref{V1} follows form \eqref{B3}  for some different constant $k>0$.

  For $V$ in Corollary \ref{C2.4}(2), \eqref{AB4'} holds for any $n\ge 1$ and $\aa(r)=  {c r^{-\ff 1 \tau}}$ for some constant $c>0$ as shown in the proof of Corollary \ref{C2.4}(2), so that $\ff{K_\psi(r)}{\aa(r)}=\ff{\kk m }{c} r^{m-1+\tau^{-1}}.$ Then
  \eqref{V2} follows from \eqref{B4} for  {$\theta= m-1+\tau^{-1}.$}

   For $V$ in Corollary \ref{C2.4}(3), \eqref{AB4'} holds for any $n\ge 1$ and $\aa(r)= \exp[cr^{-\ff 1 {q-1}}]$ for some constant $c>0$, so that $\ff{K_\psi(r)}{\aa(r)}\ge \exp[-c'r^{-\ff 1 {q-1}}]$  for some constant $c'>0$. Then
  \eqref{V3} follows from \eqref{B3}.
  \end{proof}

 \subsection{Estimate on the  convergence rate using  $K_\phi(s)+K_\psi(r)$.}

 Since $\DD-\nn V$ has Neumann spectral gap on bounded balls, we may  choose a decreasing function $\ll: [1,\infty)\to (0,\infty)$   such that
 \beq\label{ABS} \mu(1_{B_s} f^2)\le \ff 1 {\ll_s} \mu(1_{B_s}|\nn f|^2),\ \ \ f\in C_b^1(B_s),\ \mu(1_{B_s} f)=0,\ s\ge 1.\end{equation}
Moreover,     let
\beg{align*}
&g_s:= \inf\big\{\rr_V(x): |x|\le s\big\},\ \ \
  h_s:= \mu\big(|\cdot|>s \big),\ \ s>0,\\
  &  s(r):=\inf\big\{s\ge 1:\ h_s\le r\big\}, \ \ \ r>0.\end{align*}

\beg{cor}\label{C5.1} Assume   $(A_2)$ and let $\eqref{ABS}$ hold for some decreasing $\ll_\cdot$.
Then $\eqref{E1}$ has  a distributional solution with initial value $\rr_0$ such that
$$ \mu(|u_t-1|^2)\le \big\|u_0\big\|_\infty^2 \inf_{r>0} \Big\{\e^{- t\ll_{s(r)} [K_\phi(rg_{s(r)})+K_\psi(r)]} +5  r\Big\},\ \ \ t>0.$$
If $\ll_{s(r)} [K_\phi(rg_{s(r)})+K_\psi(r)]\ge \dd r^\theta$ for $r\in (0,1]$ and some constants $\dd,\theta>0$, then there exists a constant $k>0$ such that
$$\mu(|u_t-1|^2)\le k \big\|u_0\big\|_\infty^2  t^{-\ff 1 \theta},\ \ t>0.$$
\end{cor}

\beg{proof} We use the argument in the proof of Corollary  \ref{C4.1}.
 {For $f\in C_b^1(\R^d)$ with $f\ge 0$ and  {$\mu_n(f)=1$},   we have
 {$$f^2\le (f\lor r)^2 1_{B_s} +\|f\|_\infty^2 1_{B_s^c},\ \ \ s\ge 1,\ r\in (0,1).$$}
So,
\beg{align*}&  \mu_n ( |f -1|^2)=\frac{\mu(f^21_{B_n})}{\mu(B_n)}-1  \le \frac{\mu\big((f\lor r)^2 1_{B_{s\wedge n}}\big)+\|f\|_\infty^2 h_s}{\mu(B_n)} -1\\
&\le \frac{{\rm Var}_{\mu}\big((f\lor r)1_{B_{s\wedge n}} \big) +\|f\|_\infty^2 h_s}{\mu(B_n)}+\mu_n (f\lor r)^2 \mu(B_n)-1\\
&\le  \frac{{\rm Var}_{\mu}\big((f\lor r)1_{B_{s\wedge n}} \big) +\|f\|_\infty^2 h_s}{\mu(B_n)} + 3r,\ \ r>0,\ s>0, \end{align*}
where the last step is due to $\mu_n(f\lor r)^2\le (1+r)^2\le 3r+1$ for $r\in (0,1]$.
Combining this with  $\|f\|_\infty\ge 1$,    and the definitions of $K_\psi, K_\phi$ and $\ll_s$ and the fact that $\lambda_s$ is decreasing in $s$, we derive for large enough $n\geq 1$ satisfying $\mu(B_n)\geq \frac{1}{2}$,
\beg{align*}&{\rm Var}_{\mu_n}(f) \le \ff 1 {\ll_{s\wedge n}} \mu_n\big( |\nn f|^2 1_{\{f\ge r\}\cap B_s  } \big) + (4r +2h_s) \|f\|_\infty^2\\
&\le \ff{ 1}{\ll_s [K_\phi(rg_s) +K_\psi(r)]}\mu_n\big( (\phi'(f\rr_V)+\psi'(f))|\nn f|^2\big)+ (4r +2h_s) \|f\|_\infty^2\end{align*} for any $r\in (0,1),\ s>0.$
Taking $s=s(r)$ such that $h_{s(r)}\le \frac{r}{2}$, }we arrive at
$${\rm Var}_{\mu_n}(f) \le \ff { 1} {\ll_{s(r)}[K_\phi(rg_{s(r)})+K_\psi(r)]} \mu_n \big(  (\phi'(f\rr_V)+\psi'(f))|\nn f|^2\big)+ 5r  \|f\|_\infty^2$$ for $ r\in (0,1]. $ This implies  \eqref{A4} for
 $$\aa(r):= \ff { 1} {\ll_{s(r/5)}[K_\phi(rg_{s(r/5)})+K_\psi(r/5)]}.$$   So,
   the desired assertion follows from Theorem \ref{TT2}(2).
\end{proof}

 To apply Corollary \ref{C5.1}, we present an estimate on $\ll_s$, which follows form the two displays after  (3.10) in the proof of  \cite[Theorem 3.4]{RW01}.

 \beg{lem}\label{L1} Let $\gg\in C([0,\infty))$ and $r_0>0$ such that
 $$L_V|\cdot|(x) \le \gg(|x|),\ \ \ |x|\ge r_0$$
  {for $L_V=\Delta-\nabla V\cdot\nabla$.}
 Define
 \beg{align*}& \tt\psi(s):=\inf_{r\in [r_0,s]} \big[-\gg(r)],\ \ s>r_0,\\
 &\tt\phi(s):= \inf\big\{r\ge r_0:\ \tt\psi(s)^2s^2 \ge 9\mu(B_r^c)\big\},\ \ s>r_0,\\
 &\dd_s(V):= \sup\big\{V(x)-V(y):\ x,y\in B_s\big\},\ \ s>0,\end{align*}
 where $\inf\emptyset:=\infty$.
 Then there exists a constant $k>0$ such that
 \beq\label{NM} \ll_s\ge k\tt\psi(s)^2 \e^{-\dd_{\tt\phi(s)}(V)},\ \ \ s>r_0.\end{equation}
 \end{lem}

 Since  examples for $K_\psi(r)>0$ have been addressed in the last subsection, below we only illustrate Corollary \ref{C5.1} by an example with $\psi=0$.

\beg{cor}   Assume   $(A_2)$ with   $\psi=0$ and $\phi(s)= \kk s^m$ for some constants $\kk>0, m>1$. Then $\eqref{E1}$ has  a distributional solution with initial value $\rr_0$ such that  the following assertions hold.
 \beg{enumerate} \item[$(1)$] Let
 $$V(x)= W(x) + \bb |x|^q$$
 for some $W\in C_b(\R^d)$ and constants $\bb, q>0.$
   If $q\ge 1$, then  there exists a constant $k >0$ such that
 \beq\label{D1} \mu(|u_t-1|^2)\le k \big\|u_0\big\|_\infty^2 t^{-\ff 1 {2(m-1)}},\ \ t>0.\end{equation}
 If $q\in (0,1)$, then for any $\theta>2(m-1)$, there exists a constant $k >0$ such that
 \beq\label{D2} \mu(|u_t-1|^2)\le      k  \big\|u_0\big\|_\infty^2 t^{- \ff 1 {\theta}},\ \ \ t> 0. \end{equation}
 \item[$(2)$]    Let $V(x)=W(x)+ (d+q) \log (1+|x|)$ for some constant $q>0$ and  $W\in C_b(\R^d)$. Then there exists a constant $k >0$ such that
 \beq\label{D3} \mu(|u_t-1|^2)\le      k  \big\|u_0\big\|_\infty^2 t^{- \ff q{2+ (m-1)(2q+d)}},\ \ \ t> 0. \end{equation}
     \item[$(3)$]        Let $V(x)=W(x)+  d  \log (1+|x|) +q\log\log (\e+ |x|)$ for some constant $q>1$ and  $W\in C_b(\R^d)$. Then there exists a constant $k >0$ such that
 \beq\label{D4} \mu(|u_t-1|^2)\le      k  {\|\rr_0/\rr_V\big\|_\infty^2}[\log(1+ t)]^{1-q},\ \ \ t> 0. \end{equation}
       \end{enumerate}\end{cor}
 \beg{proof}   Since the Poincar\'e inequality on $B_s$ is stable under the bounded perturbation $W$, we may and do assume $W=0.$
Let $L_V:=\DD-\nn V\cdot\nn$.

 (1) When $q\ge 1$, we have
 $$\lim_{|x|\to\infty} L_V |x|= \lim_{|x|\to\infty}\Big(\ff{d-1}{|x|} -  \bb q |x|^{q-1} \Big)<0.$$
Then there exists $r_0,\kk_0>0$ such that
$$L_V |x|\le -\kk_0,\ \ \ |x|\ge r_0.$$
Choose $\gg\equiv -\kk_0$, we find constants $\kk_1,\kk_2>0$ such that $\tt\psi(s)\ge \kk_1, \tt\phi(s)\le \kk_2$ for $s\ge r_0.$
So, by \eqref{NM} and  {$\inf_{s\in (0,r_0]}\ll_s>0$}, we find a constant $\ll>0$ such that
$$\ll_s\ge   \ll,\ \ s>0.$$
 Since $\mu(\d x)= c\e^{-\bb |x|^q}\d x$ for some constants $c>0$ and $q\ge 1$, there exists a constant $k_1 >0$ such that
 $$h_s\le \e^{k_1 -  \bb s^q},\ \ \ s>0,$$ so that
\beq\label{K2} s(r):=\inf\{s>0: h_s\le r\}\le \Big[\ff{1}\bb(k_1 +\log r^{-1})\Big]^{\ff 1 q}.\end{equation}
 Then there exist  constants $k_2>0, k_3 >0$ such that
\beq\label{K3} u_{s(r)}\ge \e^{-k_2-\bb s(r)^q} \ge k_3  r.\end{equation}  Thus, there exists a constant $k_4 >0$ such that
 $$\ll_{s(r)} K_\phi(ru_{s(r)}) \ge  \ll  m\kk\big(ru_{s(r)}\big)^{m-1}\ge k_4  r^{2(m-1)}.$$
Therefore,   \eqref{D1}  follows from Corollary \ref{C5.1}.

 Now, let $q\in (0,1)$. We find  constants $r_0,\kk_1>0$ such that
 $$(L_V|\cdot|)(x)\le -\kk_1  |x|^{q-1},\ \ \ |x|\ge r_0.$$
 So,   $\tt\psi(s)\ge \kk_1 s^{q-1}$ and by $\mu(B_r^c)\le \ff c {(1+r)^q}$ for some constant $c>0$, we find a constant $\kk_2>0$ such that
 $$ \tt\phi(s) \le \kk_2 s^{-2},\ \ \ s>r_0.$$
 So, by \eqref{NM} and noting that $\dd_{\kk_2 s^{-2}}(V) $ is bounded for $s>r_0$, we find a constant $\kk_3>0$ such that
 $$\ll_s\ge \kk_3 s^{2(q-1)},\ \ \ s>r_0.$$
 Since $\inf_{s\in (0,r_0]}\ll_s>0$  {and $q\in(0,1)$}, we find a constant $\ll>0$ such that
   $$ \ll_s\ge       \ll (1+s)^{2(q-1)},\ \  {s>0}.$$
On the other hand, for any $\vv>0$ we find a constant $k_1>0$ such that
$$h_s\le \e^{k_1 -  \ff{\bb}{1+\vv} s^q},\ \ \ s>0,$$
 so that
 $$1+s(r):=1+\inf\{s>0: h_s\le r\}\le  {1+\Big[\ff{1+\vv}\bb(k_1 +\log r^{-1})\Big]^{\ff 1 q}}.$$
Then as  explained above, we find   constants $k_2,k_3>0$ such that
\beg{align*}&\ll_{s(r)} K_\phi(ru_{s(r)}) \ge  \ll (1+s(r))^{2(q-1)} m\kk\big(ru_{s(r)}\big)^{m-1}\\
&\ge k_2  [\log(1+r^{-1})]^{\ff{2(q-1)}q} r^{2(\vv+1)(m-1)}\ge k_3 r^{2(2\vv+1)(m-1)} ,\ \ r\in (0,1].\end{align*}
Therefore,   \eqref{D2}  follows from Corollary \ref{C5.1} by taking $\vv>0$ such that $2(2\vv+1)(m-1)=\theta.$

(2)  For $V(x)= (d+q)\log (1+|x|)$,   we find a constant $r_0>0$   such that
$$(L_V|\cdot|)(x)=\ff{d-1}{|x|}-\ff{d+q}{1+|x|}\le -\ff 1{1+|x|},\ \ \ |x|\ge r_0.$$
Then $\tt\psi(s)\ge  (1+s)^{-1}$ and $\tt\phi(s)$ is bounded for $s\geq r_0$, hence \eqref{NM} implies
\beq\label{90} \ll_s\ge \ll(1+s)^{-2},\ \ \ s>0\end{equation}
for some constant $\ll>0$.

On the other hand,
there exists   constants $\kk_1,\kk_2>0$ such that
$$s(r)\le \kk_1 r^{-\ff 1 q},\ \ \ u_{s(r)}\ge \kk_2 r^{\ff{d+q}q},\ \ r\in (0,1].$$
Combining this with \eqref{NM},  we find  a constant $\kk_3>0$ such that
$$\ll_{s(r)} K_\phi(ru_{s(r)}) = \ll_{s(r)} m\kk  (ru_{s(r)})^{m-1} \ge \kk_3 r^{\ff{2+(m-1)(2q+d)}q},\ \ r\in (0,1].$$
Therefore,   \eqref{D3}  follows from Corollary \ref{C5.1}.

(3) In this case, we find constants $r_0>0$   such that
$$(L_V|\cdot|)(x) \le  {-\ff 1 {2(1+|x|)}},\ \ \ |x|\ge r_0,$$
so that as in (2) we find a constant $\ll>0$ such that  \eqref{90} holds.

On the other hand, there exist constants $\kk_1,\kk_2,\kk_3>0$ such that
$$h_s\le \kk_1 [\log(\e+s)]^{1-q},\ \ \ s(r)\ge \e^{\kk_2 r^{-\ff 1 {q-1}}},\ \ \ r\in (0,1],$$
so that for some constants $\kk_4 >0$,
$$\ll_{s(r)}K_\phi(ru_{s(r)})\ge   \e^{-\kk_4 r^{-\ff 1{q-1}}},\ \ \ r\in (0,1].$$
Then \eqref{D4} follows from Corollary \ref{C5.1}.
 \end{proof}


\begin{thebibliography}{10}
\bibitem{A} S. Aida, \emph{Uniformly positivity improving property, Sobolev inequalities and spectral gap,} J. Funct. Anal. 158(1998), 152-185.
\bibitem{BE} D. Bakry, M. Emery, \emph{ Hypercontractivit\'e de semi-groupes de diffusion,} C. R.
Acad. Sci. Paris S\'er. I Math. 299(1984), 775-778.
\bibitem{BGL} D. Bakry, I. Gentil, M. Ledoux, \emph{Analysis and Geometry of Markov Diffusion Operators,} Springer-Verlag, 2014, Berlin.
\bibitem{Ba} V. Barbu,   \emph{Nonlinear Fokker-Planck equations with singular integral drifts and McKean-Vlasov SDEs,}  J. Evol. Equ. 26(2026),  1-26.
\bibitem{BR2020} V. Barbu, M.  R\"ockner, \emph{From nonlinear Fokker-Planck equations to solutions of distribution dependent SDE,}  Ann. Probab.      48(2020), 1902-1920.
\bibitem{BR2023} V. Barbu, M.  R\"ockner,  \emph{The evolution to equilibrium of solutions to nonlinear Fokker-Planck equation,} Indiana Univ. Math. J. 72(2023), 89-131.
\bibitem{BR2023JFA} V. Barbu, M.  R\"ockner,  \emph{Uniqueness for nonlinear Fokker-Planck equations and for McKean-Vlasov SDEs: the degenerate case,} J. Funct. Anal. 285(2023), Paper No. 109980, 37 pp.
\bibitem{BR2024LNM} V. Barbu, M.  R\"ockner, \emph{Nonlinear Fokker-Planck flows and their probabilistic counterparts,} Lecture Notes in Math. 2353, Springer, 2024, ix+212 pp.
\bibitem{DS} J.-D. Deuschel, D. W. Stroock, \emph{Large Deviations, } Acadeic Press, Boston 1989.
\bibitem{DS2} J.-D. Deuschel, D. W. Stroock, \emph{Hypercontractivity and spectral gap of symmetric
diffusion with applications to the stochastic Ising models,} J. Funct. Anal. 92(1990), 30-48.
\bibitem{DGGW} J. Dolbeault, I. Gentil, A. Guillin, F.-Y. Wang, \emph{$L^q$-functional inequalities and weighted porous media equations,} Potential Anal. 28(2008),  35-59.
    \bibitem{LSU}
O. A. Ladyzenskaja, V.A. Solonnikov, N. N. Ural'ceva, Linear and Quasilinear Equations of
Parabolic Type. Translated from the Russian by Smith, S. Translations of Mathematical Monographs,
vol. 23. American Mathematical Society, Providence, RI (1967).

\bibitem{R} M. Rehmeier, Weighted $L^1$-semigroup approach for nonlinear Fokker-Planck equations and generalized Ornstein-Uhlenbeck processes, \emph{arXiv:2308.09420}.
\bibitem{RW01} M. R\"ockner, F.-Y. Wang, \emph{Weak Poincar\'{e} inequalities and $L^2$-convergence rates of Markov semigroups,} J. Funct. Anal. 185(2001), 564-603.	
\bibitem{W99} F.-Y. Wang, \emph{Existence of the spectral gap for elliptic operators, } Arkiv  Math. 37(1999), 395-407.

 \bibitem{Wbook} F.-Y. Wang, \emph{Functional Inequalities, Markov Semigroups and Spectral Theory,}    Science Press/Elsevier, 2005, Beijing/New York.
 \bibitem{W13}  F.-Y. Wang, \emph{  Harnack Inequality for Stochastic Partial Differential Equations,}   Springer, Berlin, 2013.
 \bibitem{W09} F.-Y. Wang,  \emph{Log-Sobolev inequalities: different roles of Ric and Hess,}  Ann. Probab.  37(2009), 1587-1604.

\end{thebibliography}
\end{document}